\RequirePackage{fix-cm}
\documentclass[smallextended]{svjour3}       % onecolumn (second format)
\smartqed  % flush right qed marks, e.g. at end of proof
\usepackage{graphicx}
\usepackage{booktabs,multirow}
\usepackage{amssymb,amsmath}
\usepackage[all]{xy}
\usepackage{float}
\usepackage{subfig}
\usepackage{graphicx}
\usepackage[overload]{empheq}
\usepackage{rotating}
\usepackage{lscape}

\usepackage{framed} 
\usepackage{lastpage}
\usepackage[algoruled]{algorithm2e}

\newcommand{\bm}{\mathbf}
\newtheorem{assumption}{Assumption}

\begin{document}

\title{A new method based on the bundle idea and gradient sampling technique for minimizing nonsmooth convex functions%\thanks{Grants or other notes
%about the article that should go on the front page should be
%placed here. General acknowledgments should be placed at the end of the article.}
}
%\subtitle{Do you have a subtitle?\\ If so, write it here}

\titlerunning{A bundle-gradient sampling method for minimizing nonsmooth convex functions}        % if too long for running head

\author{M. Maleknia   \and  M. Shamsi }

\institute{Morteza Maleknia (ORCID iD: 0000-0001-8576-3119) \at
	Amirkabir University of Technology\\ 
	Tehran, Iran\\
	\email{m.maleknia@aut.ac.ir}
	\and
	M. Shamsi (ORCID iD: 0000-0002-3806-9316),  Corresponding author  \at
	Amirkabir University of Technology\\ 
	Tehran, Iran\\
	\email{m\_shamsi@aut.ac.ir}
}

\date{Received: date / Accepted: date}

\maketitle

\begin{abstract}
In this paper, we combine the positive aspects of the Gradient Sampling (GS) and bundle methods, as the most efficient methods in nonsmooth optimization, to develop a robust method for solving unconstrained nonsmooth convex optimization problems.
The main aim of the proposed method is to  take advantage of both GS and bundle methods, meanwhile avoiding their drawbacks.
At each iteration of this method, to find an efficient descent direction, the GS technique is utilized for constructing a local polyhedral model for the objective function. 
If necessary, via an iterative improvement process, this initial polyhedral model is improved  by some techniques inspired by the bundle and GS‌ methods. 
The convergence of the method is studied, which reveals the following positive features
(i) The convergence of our method is independent of the number of gradient evaluations required to establish and improve the initial polyhedral models. Thus, the presented  method needs much fewer gradient evaluations in comparison to the original GS method.
(ii) As opposed to GS type methods,  the objective function need not be continuously differentiable on a   full measure open set in $\mathbb{R}^n$ to ensure the convergence for the class of convex problems.
Apart from the mentioned advantages, by means of numerical simulations, we show that the presented method provides promising results in comparison with GS methods, especially for large scale problems. 
Moreover, in contrast with bundle methods, our method is not very sensitive to the accuracy of supplied gradients.
%-----
\end{abstract}

\section{Introduction}
In this paper, we consider the following minimization problem 

\begin{equation}\label{MAIN-PROBLEM}
\min f(\bm x) \quad \text{s.t.} \quad \bm x\in \mathbb{R}^n,
\end{equation}
where $f:\mathbb{R}^n\to\mathbb{R}$ is a convex function, but not necessarily differentiable. Such problems appear in many applied fields, such as data analysis, image processing, optimal control, biology, chemistry, and computational physics \cite{Makela_book,image,Poly-approx,bio-ch2,data-mining}. In this regard, it is worthwhile to develop efficient methods for solving these types of problems.

There are various methods for locating a minimizer of problem \eqref{MAIN-PROBLEM}. The subgradient method, initially developed by N. Shor \cite{shorbook}, is one of the simplest methods for solving a nonsmooth convex optimization problem. Because of its simple structure, it is still a widely used method, although it suffers from some serious drawbacks, such as lack of descent, slow convergence, and lack of practical termination criterion.

As another class of methods for solving nonsmooth optimization problems, we can refer to bundle methods as one of the most efficient methods in nonsmooth optimization, which can handle both convex and nonconvex objectives. The basic idea behind these methods is to keep the memory of the previously computed subgradients to build up a \emph{polyhedral} (piecewise linear) model for the objective function. If this polyhedral model is an adequate approximation to the objective function, by solving a quadratic subproblem, one can obtain a descent direction, which makes a substantial reduction in the objective function, and consequently, a \emph{serious step} occurs. Otherwise, the method accepts a \emph{null step} to enrich the polyhedral model by gathering more subgradient information. One drawback of the bundle type methods is that, at each iteration, a subgradient is added to the bundle of information, which posses serious difficulties with storage and the size of the corresponding quadratic subproblem. To overcome this difficulty, in \cite{Aggregate-sub}, Kiwiel developed the concept of \emph{aggregated subgradient}, through which one can control the size of the bundle, and as a result, keep the number of constraints in the quadratic subproblems bounded. Currently, the aggregation strategy is the cornerstone of many efficient methods in nonsmooth optimization (see, for instance \cite{Makela_book,kiwielbook,BT-method,Haarala,Proximal-Bundle,Haarala,Bundle-semidefinite}). In particular, based on this strategy, Schramm and Zowe developed one of the most effective variants of bundle methods, namely the Bundle Trust (BT) method \cite{BT-method}. The BT method is in continuation of the works of Kiwiel in \cite{Aggregate-sub} and \cite{Proximal-Bundle}, which adds some properties of the trust region philosophy to the bundle concept. Numerous numerical experiments have confirmed the efficiency of this method for solving a broad range of nonsmooth optimization problems (see, for instance \cite{BT-method,Bagirov2014,Zowe-book}).
%------------------

%=========== Why Gs is burned?
Generally, in the class of subgradient and bundle methods, after a few numbers of iterations, the method approaches a nonsmooth curve. In this situation, only accurate subgradients can  contribute to the efficiency of the method. This fact is a major limitation for these methods, especially when we deal with real-life and black-box problems, in which  supplying only one accurate subgradient is not an easy task.
%----------
To resolve this drawback, Burke et al. \cite{Burke2005}  developed the Gradient Sampling (GS) method to solve a broad range of nonsmooth optimization problems without explicit computation of subgradients. 
Soon after, the original GS approach  was considered as a ground for developing various types of GS methods (see \cite{Curtis2013,Curtis2012,Fast-GS,GS-full,CG-GS,Kiwiel2007,Kiwiel2010,Diff-Check}).
%--------------
The GS methods can deal with a variety of nonsmooth problems using only approximate gradients. Therefore, whenever function evaluations are not expensive, the GS type methods can be the algorithm of choice. 
%----------------$

In the family of the GS methods, through a uniform distribution,  a set of auxiliary points is randomly generated from an $\varepsilon$-neighborhood of the current point. 
We know that, with probability 1, a uniform random choice of points gives a set of differentiable points in $\mathbb{R}^n$. Moreover, due to the uniform distribution, the set of the sampled points likely lies on the opposite sides of the discontinuity of the gradient map.  
%--------
Consequently, significant knowledge of the nearby nonsmooth curve is obtained by computing the gradient of the objective function on the sampled points. Then, the GS type methods use this key knowledge to create an efficient  decent direction in order to make a remarkable progress towards a stationary point.

%=========== Survey
The original GS method \cite{Burke2005,GS-full} is a descent method to minimize locally Lipschitz functions, which are continuously differentiable on a full measure open set in $\mathbb{R}^n$. This method, at each iteration, tries to approximate the $\varepsilon$-steepest descent direction. For this purpose, it evaluates the gradient of $f$ at some randomly generated points, and then the convex hull of this information is used to obtain an approximation of $\varepsilon$-steepest descent direction. The convergence theory of the GS method was first developed in \cite{Burke2005}. Through some subtle modifications, Kiwiel provided stronger convergence results in \cite{Kiwiel2007}.
Soon after, these results became the theoretical foundation of many variants of the GS type methods \cite{Curtis2012,Curtis2013,Diff-Check,Kiwiel2010}. The GS methods are robust and can be applied to a wide range of nonsmooth problems, including non-Lipschitz functions. However, to observe a good practical behavior, we may need to increase the size of the sample to $2n$, which makes these methods computationally extensive for solving large scale problems. 

In this paper, using the GS technique, we propose a version of bundle methods in which no explicit computation of  (sub)gradient information is needed. 
The main idea of our method is to use the GS technique to construct a local polyhedral approximation for the objective function $f$. 
For this purpose, at each iteration,  a region around the current point, namely \emph{sampling region}, is defined and a set of points is independently and uniformly sampled from this region. 
Then, we evaluate the gradient of the objective function on the sampled points to obtain a local polyhedral approximation to the objective function.
Indeed, the sampled points can be considered as auxiliary points that are used to build up our initial local approximation to $f$. 
%----------------------
Using this initial model, a quadratic subproblem is established  and solved to generate a search direction. 
If there is a good agreement between the objective function and model on the sampling region, then the generated search direction leads to a reduction in the objective function.
However, due to the kinky structure of nonsmooth functions, this initial model may not be  well enough to generate an efficient descent direction, and consequently, we may need to improve the quality of the model. To this end, we perform an iterative process, namely the \emph{improvement process}, in which we improve the quality of the initial model either by making the sampling region smaller or by gathering more gradient information. In addition, to control the size of the subproblems throughout the improvement process, we use the \emph{aggregation strategy} proposed by Kiwiel in \cite{Aggregate-sub}. 
Furthermore, an adaptive gradient selection strategy is proposed to preserve the most active gradients throughout the improvement procedure.
%----------------
As opposed to the bundle methods, once a serious step occurs, we do not save the gradient information computed at the previous iteration. Indeed, at the next iteration, we construct a new local polyhedral model using a completely new set of sampled points. 

%----Advantegeous
The presented method inherits the advantageous properties of GS and bundle methods; at the same time, it circumvents the drawbacks of these methods. 
%---
In comparison with GS type methods, our method requires much fewer gradient evaluations. 
More precisely, only $\mathcal{O}(1)$ gradient evaluation is sufficient to establish the initial polyhedral model and, if necessary, a single gradient evaluation at each iteration of the improvement process. 
%--------
As an impressive feature of our method, as opposed to bundle methods, the gradients can be supplied using simple difference approximation formulas. As a consequence, the method can efficiently solve those problems for which explicit computation of (sub)gradients is cumbersome.
%----
Moreover, the proposed algorithm offers some user-defined parameters, which allow the user to customize the method based on the structure of a problem.

%--Sectionology
This paper is organized as follows. Section \ref{Preliminaries} provides some mathematical preliminaries used in this article. A comprehensive description of the method is presented in Section \ref{Method-Description}. The convergence analysis of the proposed method is studied in Section \ref{Convergence analysis}. We report the results of numerical experiments in Section \ref{Numerical results}, and Section \ref{Conclusion} concludes the paper.

\section{Preliminaries }\label{Preliminaries}
The following notations are used in this article. $\mathbb{R}^n$ stands for the $n$-dimensional Euclidean space and its inner product is denoted by $\langle \bm x, \bm y\rangle:=\sum_{i=1}^n x_iy_i$, which induces the Euclidean norm $\lVert\bm x\rVert:=\langle\bm x,\bm x\rangle^{1/2}$. Moreover, $B(\bm x, \varepsilon):=\{\bm y\in\mathbb{R}^n \,:\, \lVert\bm x-\bm y\rVert\leq \varepsilon\}$ defines a closed ball centered at $\bm x$ with the radius $\varepsilon>0$.

The classical directional derivative of a function $f:\mathbb{R}^n\to\mathbb{R}$ at $\bm x\in\mathbb{R}^n$ in the direction $\bm d\in\mathbb{R}^n$ is defined as \cite{Bagirov2014}
$$f'(\bm x;\bm d):=\lim_{t\downarrow 0}\, t^{-1}[f(\bm x+t\bm d)-f(\bm x)]. $$
If $f$ is convex, then for a point $\bm x\in\mathbb{R}^n$, the directional derivative exists in every direction $\bm d\in\mathbb{R}^n$ \cite{Bagirov2014}. Let
$$D:=\{\bm x\in\mathbb{R}^n \,\,:\,\, f\,\  \text{is differentiable at}\, \bm x \}, $$
be a subset of $\mathbb{R}^n$ where the objective function $f$ is differentiable. We recall that every convex function is locally Lipschitz. Therefore, by Rademacher's theorem \cite{Evans2015}, every convex function is differentiable almost everywhere. Thus, whenever we independently and uniformly sample some points from a subset of $\mathbb{R}^n$ with positive measure, the sampled points lie in $D$, with probability 1. This key property plays a vital role in the theory of the GS type methods.

The subdifferential of a convex function $f$ at the point $\bm x\in\mathbb{R}^n$ is given by~\cite{Bagirov2014}
\begin{equation}\label{sub}
\partial f(\bm x):=\{\boldsymbol{\xi} \in\mathbb{R}^n \,:\, f(\bm y)\geq f(\bm x)+\boldsymbol{\xi}^t(\bm y-\bm x)\,\,\,  {\rm{for \, all} }\,\,\, \bm y\in\mathbb{R}^n\}.
\end{equation} 
In addition, for $\epsilon\geq 0$, the $\epsilon$-subdifferential, which is the generalization of the ordinary subdifferential, is defined by \cite{Bagirov2014}
\begin{equation}\label{epsilon-sub}
\partial_{\epsilon} f(\bm x):=\{\boldsymbol{\xi} \in\mathbb{R}^n \,:\, f(\bm y)\geq f(\bm x)+\boldsymbol{\xi}^t(\bm y-\bm x)-\epsilon\,\,\,  {\rm{for \, all} }\,\,\, \bm y\in\mathbb{R}^n\}.
\end{equation} 
For every $\bm x\in\mathbb{R}^n$ and $\epsilon\geq 0$, it is easy to see that
$\partial f(\bm x)\subset \partial_{\epsilon} f(\bm x)$, and for $\varepsilon=0$ we have $\partial f(\bm x)= \partial_{\epsilon} f(\bm x)$.
It is also shown in \cite{Bagirov2014} that for every $\epsilon\geq 0$ the set valued map $\partial_{\epsilon} f:\mathbb{R}^n\rightrightarrows\mathbb{R}^n$ is upper semicontinuous and the set $\partial_{\epsilon} f(\bm x)$ is a nonempty, convex and compact subset of $\mathbb{R}^n$.

Finally, we recall that $\bm x\in\mathbb{R}^n$ is a minimizer for the convex function $f$, if and only if \cite{Bagirov2014}
$$\bm 0\in \partial f(\bm x).$$

\section{The presented Bundle-Gradient Sampling (B-GS) method}\label{Method-Description}
In this section, the general framework of the proposed method is presented. To derive an efficient method for solving nonsmooth convex optimization problems, we develop a bundle based method using the GS technique.

In this method, at each iteration, we define a region around the current point within which some gradient information is sampled to construct a \emph{local polyhedral model} of the objective function. This polyhedral model is considered as a local approximation to the objective function $f$. This model is used to establish a strictly convex \emph{Quadratic optimization Problem} (QP), which its solution provides a search direction. Then, a standard \emph{sufficient decrease condition} checks the efficiency of this direction.  If our polyhedral model is an adequate approximation of the objective function $f$, the generated search direction is an efficient descent direction, which leads to a serious step. Otherwise, we need to enhance the quality of the model. For this purpose, we carry out an \emph{improvement process}, which iteratively enhances the quality of the polyhedral model. At each iteration of this process, two strategies may be used. The first strategy is to enrich the model by adding a new gradient, and the second one is to reconstruct the polyhedral model through a smaller sampling region. To find out which strategy works better, we use a simple criterion that was first introduced in \cite{BT-method}. In addition, to control the size of the QPs in the improvement process, we employ the \emph{aggregation strategy} proposed in \cite{Aggregate-sub}. 

In the rest of this section, different parts of the proposed method are thoroughly described.

\subsection{Constructing a local polyhedral approximation using the GS technique}\label{subsec-3.1}
In this subsection, we employ the GS technique to establish a local polyhedral approximation for the objective function $f$. 
For this purpose, suppose that we are at the $k$-th iteration of the algorithm, and $\mathbf{x}_k\in\mathbb{R}^n$ at which $f$ is differentiable is the current point. Furthermore, let $B(\bm x_k, \varepsilon_k)$ be the current sampling region, where $\varepsilon_k>0$ is called \emph{sampling radius}. Then, for a given \emph{sample size} $m\in\mathbb{N}$, we sample $\bm{s}_{k,1},\ldots,\bm s_{k,m}$ independently and uniformly from the sampling region $B(\bm x_k,\varepsilon_k)$.
As discussed in Section \ref{Preliminaries}, the objective function $f$ is differentiable at the sampled points $\bm{s}_{k,1},\ldots,\bm s_{k,m}$, with probability 1. Therefore, we proceed with assuming that the gradient of $f$ is defined at the sampled points. Moreover, due to some technical reasons, we add $\bm s_{k,0}:=\bm x_k$ to the set of sampled points. Now, using the gradient information at these points, one can define the following linearizations 
\begin{equation}\label{linearizations}
f_{k,j}(\bm x):=f(\bm s_{k,j})+\bigl\langle\nabla f(\bm s_{k,j}), \bm x- \bm s_{k,j}\bigr\rangle, \qquad j=0,1,\ldots,m.
\end{equation}
Let $e_{k,j}$ be the error of the $j$-th linearization at the current point $\bm x_k$, i.e.,
\begin{align}
e_{k,j}:&=f(\bm x_k)-f_{k,j}(\bm x_k)\nonumber\\&
=f(\bm x_k)-\left[f(\bm s_{k,j})+\bigl\langle\nabla f(\bm s_{k,j}), \bm x_k- \bm s_{k,j}\bigr\rangle\right], \qquad j=0,1,\ldots,m. \label{error-of-linearizations}
\end{align}
Note that, convexity of $f$ implies that $e_{k,j}\geq 0$ for every $j=0,1,\ldots,m$. Furthermore, since $\bm s_{k,0}=\bm x_k$, it is easy to see that $e_{k,0}=0$. Because of convexity of $f$, for every $j=0,1,\ldots,m$, we can write
$$f(\bm x)\geq f(\bm s_{k,j})+ \bigl\langle \nabla f(\bm s_{k,j}), \bm x-\bm s_{k,j}\bigr\rangle, \qquad \text{for all} \,\, \bm x\in\mathbb{R}^n, $$
and by adding $e_{k,j}-\left[f(\bm x_k)-f(\bm s_{k,j})-\bigl\langle\nabla f(\bm s_{k,j}), \bm x_k- \bm s_{k,j} \bigr\rangle\right]=0$ to the above inequality, we conclude
\begin{equation}\label{e-nabla-inequality}
f(\bm x)\geq f(\bm x_k)+ \bigl\langle \nabla f(\bm s_{k,j}), \bm x-\bm x_k\bigr\rangle -e_{k,j}, \qquad \text{for all} \,\, \bm x\in\mathbb{R}^n,
\end{equation}
which means that
\begin{equation}\label{e-nabla-relation}
\nabla f(\bm s_{k,j})\in\partial_{e_{k,j}} f(\bm x_k), \qquad \text{for all} \,\, j=0,1,\ldots,m. 
\end{equation}
Therefore, the linearization error $e_{k,j}$ indicates how much the gradient $\nabla f(\bm s_{k,j})$ deviates from being a member of $\partial f(\bm x_k)$. 

Using the linearizations \eqref{linearizations}, at the $k$-th iteration of the method, one can  obtain the following  \emph{local polyhedral approximation} to $f$ 
\begin{equation}\label{poly-model}
f^k_p(\bm x):= \max\{f_{k,j}(\bm x)\,\, : \,\, j=0,1,\ldots,m\}.
\end{equation}
Thanks to convexity of $f$, one can see that $f^k_p$ is a lower approximation for the objective function $f$, i.e.,
$$f^k_p(\bm x)\leq f(\bm x), \qquad \text{for all}\,\, \bm x\in\mathbb{R}^n. $$
We emphasize that $f^k_p(\bm x)$ is a poor approximation to $f(\bm x)$ when $\lVert \bm x-\bm x_k\rVert$ is large, because the gradient information $\nabla f(\bm s_{k,1}),\ldots,\nabla f(\bm s_{k,m})$ are limited to the sampling region $B(\bm x_k,\varepsilon_k)$. This gives the reason why the term ``\textit{local}'' is used. We note that the sample size $m\in\mathbb{N}$ affects the quality of polyhedral model~\eqref{poly-model}. Generally, using a large sample size enhances the quality of this model but at the cost of increased storage and computation time.

\subsection{Search direction finding subproblem and sufficient decrease condition}\label{subsec-3.2}
Using the polyhedral model \eqref{poly-model}  to find a descent direction for the objective function $f$ at the current point $\bm x_k$, leads to the following \emph {search direction finding subproblem}
\begin{align}\label{Q_sub}
\begin{split}
&\min f^k_p(\bm x_k+\bm d) \\& \text{s.t.} \quad \frac{1}{2}\lVert\bm d\rVert^2\leq \varepsilon_k. 
\end{split}
\end{align}
Note that, the quadratic constraint in the above problem restricts our minimization to the sampling region $B(\bm x_k,\varepsilon_k)$, where we expect our model to be an adequate approximation of the objective function $f$. However, solving subproblem \eqref{Q_sub} accurately at each iteration of the method can be a time-consuming process. Fortunately, as we will show, only an \emph{approximate solution} ensures convergence and nice practical behavior. In this regard, we move the quadratic constraint to the objective function through a proportional penalty parameter, i.e., we consider  the following unconstrained minimization subproblem
\begin{align}\label{nonsmooth_sub}
\min_{\bm d\in\mathbb{R}^n} f^k_p(\bm x_k+\bm d)+\frac{1}{2}\varepsilon_k^{-\alpha} \lVert\bm d\rVert^2, 
\end{align}
in which $\alpha$ is a positive real number. Here, we give a quick explanation of the penalty parameter $\varepsilon_k^{-\alpha}$. Firstly, it is proportional to the radius of the sampling region; the smaller the sampling radius, the larger the penalty term. Secondly, the user-defined parameter $\alpha$ controls our sensitivity to the violation of the quadratic constraint. Indeed, since $\varepsilon_k\leq1$ in our implementations, one can be more sensitive to the violation of the quadratic constraint by choosing $\alpha>1$, while $0<\alpha<1$ behaves oppositely. Moreover, one can think of $\alpha=1$ as a neutral parameter.

Using \eqref{linearizations} and \eqref{error-of-linearizations}, one can see that
\begin{align*}
f_{k,j}(\bm x_k+\bm d)&= f(\bm s_{k,j})+\bigl\langle\nabla f(\bm s_{k,j}), \bm x_k+\bm d-\bm s_{k,j}\bigr\rangle\\&
\qquad +\left[f(\bm x_k)-f(\bm x_k)\right]\\&
=\bigl\langle\nabla f(\bm s_{k,j}),\bm d \bigr\rangle -\left[f(\bm x_k)-f(\bm s_{k,j})- \bigl\langle\nabla f(\bm s_{k,j}), \bm x_k-\bm s_{k,j}\bigr\rangle \right] \\&\qquad +f(\bm x_k)\\&
= \bigl\langle\nabla f(\bm s_{k,j}),\bm d \bigr\rangle - e_{k,j}+f(\bm x_k).
\end{align*}
Thus, in view of \eqref{poly-model}, we have 
$$f^k_p(\bm x_k+\bm d)=\max\left\{\bigl\langle\nabla f(\bm s_{k,j}),\bm d\bigr\rangle-e_{k,j}\,\,:\,\,  j=0,1,\ldots,m\right\}+f(\bm x_k), $$
and after dropping the constant $f(\bm x_k)$, subproblem \eqref{nonsmooth_sub} becomes
\begin{equation}
\min_{\bm d\in\mathbb{R}^n} \, \max_{j=0,\ldots,m}\left\{\bigl\langle\nabla f(\bm s_{k,j}),\bm d\bigr\rangle-e_{k,j}\right\}+\frac{1}{2}\varepsilon_k^{-\alpha} \lVert\bm d\rVert^2,
\end{equation}
which is \textit{not} a smooth minimization problem. To resolve this drawback, the following \textit{epigraph form} \cite{Boyd} of this problem is considered
\begin{align}\label{Main_sub}
\begin{split} 
\min_{(z,\bm d)} \ \ &  z+\frac{1}{2}\varepsilon_k^{-\alpha} \lVert\bm d\rVert^2 \\
\text{s.t.} \ \ &  
\bigl\langle\nabla f(\bm s_{k,j}),\bm d\bigr\rangle-e_{k,j}\leq z, \quad j=0,1,\ldots,m.
\end{split}
\end{align}
We emphasize that problem \eqref{Main_sub} is a QP, and its objective function is strictly convex. Hence, it has a unique optimal solution, which is denoted by $(z_k,\bm d_k)$. Note that, $z_k$ and $\bm d_k$ are related to each other as follows
\begin{equation}\label{z andf-k-p}
z_k=\max\left\{\bigl\langle\nabla f(\bm s_{k,j}),\bm d_k\bigr\rangle-e_{k,j}\,\,;\,\, j=0,1,\ldots,m\right\}.
\end{equation}
Alternatively, one can consider the dual of problem \eqref{Main_sub} which is given by
\begin{align}\label{Dual-sub}
\begin{split}
\min_{\boldsymbol{\lambda}} \ \ & \frac{1}{2} \lVert \sum_{j=0}^m \lambda_j \nabla f(\bm s_{k,j})\rVert^2 +\frac{1}{\varepsilon_k^\alpha} \sum_{j=0}^m \lambda_j e_{k,j}  \\
{\rm{s.t.}} \ \ & \sum_{j=0}^m \lambda_j=1, \quad \lambda_j\geq 0, \quad j=0,\ldots,m.
\end{split}
\end{align} 
Let $\boldsymbol{\lambda}^k=(\lambda^k_0,\ldots,\lambda^k_m)\in\mathbb{R}^{m+1}$ be the optimal solution of problem \eqref{Dual-sub}. Then some simple duality arguments relate $\boldsymbol{\lambda}^k$ to $(z_k, \bm d_k)$ as follows
\begin{subequations}\label{Rep-z-and-d}
	\begin{align}
	&z_k= - \varepsilon_k^\alpha \lVert \sum_{j=0}^m \lambda^k_j \nabla f(\bm s_{k,j})\rVert^2 - \sum_{j=0}^{m}\lambda^k_j e_{k,j},\label{Rep-z}\\&
	\bm d_k=- \varepsilon_k^\alpha \sum_{j=0}^m \lambda^k_j \nabla f(\bm s_{k,j}).\label{Rep-d}
	\end{align}	
\end{subequations}
For the sake of simplicity in notations and later use, let
\begin{equation}\label{tilde-g,e-0}
\tilde{\bm g}_k:=\sum_{j=0}^m \lambda^k_j \nabla f(\bm s_{k,j}) \quad  \text{and} \quad \tilde{e}_{k}:= \sum_{j=0}^{m}\lambda^k_j e_{k,j},
\end{equation}
then \eqref{Rep-z-and-d} is rewritten by
\begin{equation}\label{Rep-d-and-z-by-aggregate}
z_k= - \varepsilon_k^\alpha \, \lVert \tilde{\bm g}_k\rVert^2 - \tilde{e}_k \quad \text{and}\quad \bm d_k=- \varepsilon_k^\alpha \, \tilde{\bm g}_k.
\end{equation}
In addition, if we denote the optimal value of subproblem \eqref{Dual-sub} by $w_k$, then
\begin{equation}\label{optimal-value-of-Dualsub}
w_k=\frac{1}{2} \lVert \tilde{\bm g}_k\rVert^2 + \frac{1}{\varepsilon^{\alpha}_k} \, \tilde{e}_k.
\end{equation}
In the light of relations \eqref{Rep-z-and-d} (or equivalently \eqref{Rep-d-and-z-by-aggregate}), the optimal solution of dual problem \eqref{Dual-sub} yields the optimal pair of primal problem \eqref{Main_sub},  $(z_k,\bm d_k)$. Furthermore, taking a glance at \eqref{Rep-d} tells us that the optimal solution $\boldsymbol{\lambda}_k$ reveals the most active gradients in our polyhedral model. More precisely, let $\lambda^k_{(j)}$ be the $j$-th largest element of $\boldsymbol{\lambda}^k$, i.e.,
$$\lambda^k_{(1)}\geq \lambda^k_{(2)}\geq\ldots\geq \lambda^k_{(m)}\geq\lambda^k_{(m+1)}\geq 0,$$
then $\nabla f(\bm s_{k,(1)})$ has the largest weight in equation \eqref{Rep-d}, and hence, it plays the most important role in our polyhedral model. Similarly, $\nabla f(\bm s_{k,(m+1)})$ has the smallest weight in this equation, which means that this gradient information is of least importance in the polyhedral model. We shall use this key fact to control the size of subproblems in the next subsection. Based on the above discussion, it is beneficial to work with the dual problem \eqref{Dual-sub}.

To derive a variable that measures the stationarity of the current point $\bm x_k$, multiply each of \eqref{e-nabla-inequality} by $\lambda^k_j$ and sum up, to obtain the following subgradient inequality (also use \eqref{tilde-g,e-0})
\begin{equation}\label{Sub-inequality-for-tilde-g}
f(\bm x)\geq f(\bm x_k)+ \bigl\langle \tilde{\bm g}_k, \bm x-\bm x_k\bigr\rangle -\tilde{e}_k, \qquad \text{for all} \,\, \bm x\in\mathbb{R}^n,
\end{equation} 
which means that
\begin{equation}\label{e-nabla-relation-in sum}
\tilde{\bm g}_k\in\partial_{\tilde{e}_k} f(\bm x_k).
\end{equation}
Thus, if we define 
\begin{equation}\label{Opt-Tol}
v_k:=\frac{1}{2}\lVert \tilde{\bm g}_k\rVert^2 + \tilde{e}_k, 
\end{equation}
then $v_k$ measures the stationarity of the current point $\bm x_k$, because $\frac{1}{2}\lVert \tilde{\bm g}_k\rVert^2$ indicates how much $\tilde{\bm g}_k$ differs from the null vector, and $\tilde{e}_k$ measures the distance from $\tilde{\bm g}_k$ to $\partial f(\bm x_k)$. In particular, $v_k=0$ implies $\bm 0\in\partial f(\bm x_k)$, which means that $\bm x_k$ is optimal for problem \eqref{MAIN-PROBLEM}. In this regard, we shall use the value of this variable as a stopping criterion.

Now assume that $(z_k, \bm d_k)$ is in hand by solving subproblem \eqref{Dual-sub} and using relations \eqref{tilde-g,e-0} and \eqref{Rep-d-and-z-by-aggregate}. Let $\beta\in(0,1)$. Then one can check the efficiency of the search direction $\bm d_k$ by the following condition, which is called \textit{sufficient decrease condition} \cite{Aggregate-sub,kiwielbook},
\begin{equation}\label{Suff-decrease}
f(\bm x_k+\bm d_k)-f(\bm x_k)\leq \beta z_k.
\end{equation}
In the case that $\bm d_k$ does not satisfy the sufficient decrease condition \eqref{Suff-decrease}, we conclude that the polyhedral model \eqref{poly-model} is not an adequate approximation to the objective function $f$ and we need to improve the quality of this model, which is the subject of the next subsection.

For the case $\bm d_k$ satisfies condition \eqref{Suff-decrease}, it is reasonable to update $\bm x_k$ by $\bm x_{k+1}:=\bm x_k+\bm d_k$, and then, repeating the process of subsection \ref{subsec-3.1} for the new trial point $\bm x_{k+1}$. However, the point $\bm x_k+\bm d_k$ must also be a differentiable point (we recall that at the new iteration $k+1$ we have $\bm s_{k+1,0}=\bm x_{k+1}$). 
%------------------------
Therefore, following \cite{Burke2005}, instead of updating $\bm x_k$ by $\bm x_{k+1}:=\bm x_k+\bm d_k$, we set $\bm x_{k+1}:=\texttt{FDP1}(\mathbf x_k, \mathbf d_k, \beta, z_k,\sigma)$, in which  function $\texttt{FDP1}$ has been described in Algorithm \ref{Algorithm 1(FDP)}. 

\LinesNumbered
\IncMargin{1em}
\begin{algorithm}[]
	%\SetAlgoLined
	\SetAlgoNoLine
	\SetKwInOut{Input}{inputs}
	\SetKwInOut{Output}{output}
	\SetKwInOut{Required}{required}
	\SetKwFor{Loop}{Loop}{}{End\,Loop}
	
	\Indm  
	\Input{  $\bm x, \bm d\in\mathbb{R}^n$, line search parameters $\beta, z$, and the maximum allowable perturbation $\sigma>0$.  }
	\Output{$\hat{\bm x}+\bm d$ as a differentiable perturbation of $\bm x+\bm d$.  }
	\Required{$\bm x, \bm d, \beta$ and $z$ satisfy $f(\bm x+\bm d)-f(\bm x)\leq \beta z$.}
	\Indp
	\BlankLine
	
	\SetKwFunction{FMain}{FDP1}
	\SetKwProg{Fn}{Function}{:}{}
	\Fn{\FMain{$\mathbf x$, $\mathbf d$, $\beta$, $z$, $\sigma$}}{
		$\hat{\mathbf x}:=\mathbf x$, 
		$\hat{\sigma}:=\sigma$ \;
		\While{$\hat{\mathbf x}+\mathbf d\in D\  \textbf{\rm and}\  f(\bm {\hat{x}}+\bm d)-f(\bm x)\leq \beta z $}{
			Sample $\bm{\hat{x}}$ independently and uniformly from $B(\bm x,\hat\sigma)$ \;
			$\hat\sigma:=\frac{\hat\sigma}2$ \;
		}
		\textbf{return} $\hat{\bm x}+\bm d$ \; 
	}
	\textbf{End Function}
	\caption{\rule{0pt}{2ex} Finding  Differentiable Perturbation 1  (FDP1)}
	\label{Algorithm 1(FDP)}
\end{algorithm}
\DecMargin{1em} 
Algorithm \ref{Algorithm 1(FDP)} is similar to the one proposed in \cite{Burke2005}. However, the only difference is that we have added  the input parameter $\sigma$ to specify the maximum allowable perturbation, so that 
the returned point by this function lies in the $\sigma$-neighborhood of $\bm x_k+\bm d_k$. In the following, we give a quick explanation of how this algorithm works. In the case that $\bm x_k+\bm d_k\in D$, the algorithm leaves the point $\bm x_k+\bm d_k$ unchanged. For the case $\bm x_k+\bm d_k\notin D$, this algorithm starts an iterative process to find a point $\hat{\bm x}_k$ around $\bm x_k$ so that the resulting point $\hat{\bm x}_k+ \bm d_k$ lies in $D$ and also satisfies sufficient decrease condition \eqref{Suff-decrease}. Due to the continuity of $f$, this iterative process terminates after finitely many iterations, with probability 1 \cite{Burke2005,Kiwiel2007}. Updating $\bm x_k$ by  $\bm x_{k+1}:=\texttt{FDP1}(\mathbf x_k, \mathbf d_k, \beta, z_k,\sigma)$  ensures  that  

\begin{align*}
\lVert {\mathbf x_{k+1}}-(\bm x_k+\bm d_k)\rVert&=  \lVert \texttt{FDP1}(\mathbf x_k, \mathbf d_k, \beta, z_k,\sigma)-(\bm x_k+\bm d_k)\rVert
\\& =\lVert (\hat{\bm x}_k+\bm d_k)-(\bm x_k+\bm d_k)\rVert =\lVert \hat{\bm x}_k-\bm x_k\rVert\\&<\sigma.
\end{align*}
Therefore, the distance between $\bm x_k+\bm d_k$ and its differentiable perturbation, $\bm x_{k+1}$, can be controlled through the parameter $\sigma$.

According to \cite{Burke2005}, the case $\bm x_k+\bm d_k\notin D$ is unlikely to occur, and in practice we do not check this possibility. However, since it is not a zero probability event \cite{Diff-Check}, Algorithm \ref{Algorithm 1(FDP)} plays its role as a theoretical trick to ensure the convergence of the method with probability 1.

\subsection{Improvement process}\label{subsec-3.3}
In this subsection, we assume that the search direction $\bm d_k$, obtained by solving subproblem \eqref{Dual-sub}, does not satisfy sufficient decrease condition \eqref{Suff-decrease}. In such a situation, one can infer that our current polyhedral model is not an adequate approximation to the objective function $f$, and therefore, the quality of this model has to be improved. For this purpose, we carry out an iterative process, namely the improvement process. At each iteration of this procedure, we use the current inefficient search direction to obtain a new auxiliary point, at which we can define a new linearization. In the case that the error of this new linearization is small, we append it to the previous polyhedral model, and consequently, the model is enriched. Next, some strategies are used to control the size of subproblems. On the other hand, if the error of the new linearization exceeds some tolerance, instead of enriching the model, we reduce the ray of the sampling region and we reconstruct the polyhedral model \eqref{poly-model}. In this case, the improvement process is terminated and we accept a \emph {null step}. Then  the process of subsection \ref{subsec-3.1}, with a smaller sampling radius, is repeated.

Here, we explain how this improvement process works. Throughout this process, the subscript $k$ is kept fixed, and the running index is subscript $i$. At first, we fully describe the iteration $i=0$, because it provides the basis for updating variables recursively. Next, we present the general form of the subproblems we need to solve at each iteration of this process and some recursive rules to update variables. To initialize this process, set
\begin{align*}\label{Initialization}
\boldsymbol{\lambda}^{k,0}\hspace{-2pt}:=\boldsymbol{\lambda}^k,\, \tilde{\bm g}_{k,0}\hspace{-2pt}:=\tilde{\bm g}_k, \, \tilde{e}_{k,0}\hspace{-2pt}:=\tilde{e}_{k}, \bm d_{k,0}\hspace{-2pt}:=\bm d_k,\, z_{k,0}:=z_k,\,  w_{k,0}:=w_k, \, v_{k,0}:=v_k. 
\end{align*}
We know that  $\bm d_{k,0}$ does not satisfy  sufficient decrease condition \eqref{Suff-decrease}, i.e.,
\begin{equation}\label{insuff-decrease}
f(\bm x_k+\bm d_{k,0})-f(\bm x_k) > \beta z_{k,0}. 
\end{equation}
Before proceeding any further, similar to the previous section, we may need a differentiable perturbation of the point $\bm x_k+\bm d_{k,0}$. For this purpose, we use \texttt{FDP2} function which is defined similarly to Algorithm \ref{Algorithm 1(FDP)}, and the only difference is the following modifications 
\begin{align*}
&`` \bm {required}:\bm x, \bm d, \beta \,\, \text{and}\,\, z\,\, \text{satisfy}\,\, f(\bm x+\bm d)-f(\bm x)> \beta z \text{''}\\ & 
`` \bm{Line\, 9}:  \bm{while}  \ \   \bm{\hat{x}}+\bm d\in D \ \ {\rm{and}} \ \   f(\bm {\hat{x}}+\bm d)-f(\bm x)> \beta z \,\,\, \bm{do}\text{''} 
\end{align*}
Now, we define the new auxiliary point $\bm s_{k,m+1}$ as
$$\bm s_{k,m+1}:=\texttt{FDP2}(\bm x_k, \bm d_{k,0}, \beta, z_{k,0}, \sigma), $$
and then, we consider its corresponding  linearization
\begin{equation}\label{New-linearization}
f_{k,m+1}(\bm x):= f(\bm s_{k,m+1})+\bigl\langle \nabla f(\bm s_{k,m+1}), \bm x-\bm s_{k,m+1}\bigr\rangle,
\end{equation} 
with the following linearization error 
\begin{equation}\label{error-of-new-linearization}
e_{k,m+1}:= f(\bm x_k)-f_{k,m+1}(\bm x_k).
\end{equation}
Now, we need to find out whether the new linearization \eqref{New-linearization} can effectively improve the quality of the model or not. For this purpose, we use the following criterion, which initially proposed in \cite{BT-method}
\begin{equation}\label{Criterion}
e_{k,m+1}\leq \gamma \tilde{e}_{k,0} \quad \text{or} \quad \lvert f(\bm x_k+\bm d_{k,0})-f(\bm x_k)\rvert \leq\frac{1}{2} \lVert \tilde{\bm g}_{k,0}\lVert^2+\tilde{e}_{k,0},
\end{equation}
in which $\gamma$ is a user-defined parameter from $(0,1)$. Notice that the first part of \eqref{Criterion} controls the error of the new linearization, while its second part plays a technical role to guarantee the convergence of the method.
If criterion \eqref{Criterion} holds, then we enrich our polyhedral model by appending the new linearization $f^k_{m+1}$ to the previous model. Otherwise, adding the new linearization \eqref{New-linearization} does not make a significant improvement to the model. 
In this case, terminating the improvement process and reconstructing the polyhedral model \eqref{poly-model} using a smaller sampling radius would be a better strategy. In this respect, for the reduction factor $\mu\in(0,1) $,  we set
$$\varepsilon_{k+1}:=\mu\varepsilon_k \quad \text{and} \quad \bm x_{k+1}:=\bm x_k,  $$
and the process of subsection \ref{subsec-3.1} is repeated.

Now, we continue our discussion by assuming that \eqref{Criterion} holds. In this case, the new linearization \eqref{New-linearization} is added to polyhedral model \eqref{poly-model}. In mathematical terms, the following constraint is appended to the previous primal subproblem~\eqref{Main_sub}
\begin{equation}\label{New-constraint}
\bigl\langle\nabla f(\bm s_{k,m+1}),\bm d\bigr\rangle-e_{k,m+1}\leq z,
\end{equation}
and, as a result, subproblem \eqref{Main_sub}  is enriched as follows 
\begin{subequations}\label{enriched-primal-subproblem}
	\begin{align}
	\min_{(z,\bm d)} \ \ & z+\frac{1}{2}\varepsilon_k^{-\alpha} \lVert\bm d\rVert^2\nonumber\label{Aggregated-constraints} \\
	\text{s.t.} \ \ & 
	\bigl\langle\nabla f(\bm s_{k,j}),\bm d\bigr\rangle-e_{k,j}\leq z, \quad j=0,1,\ldots,m,\\&\bigl\langle\nabla f(\bm s_{k,m+1}),\bm d\bigr\rangle-e_{k,m+1}\leq z.\label{Large-subproblem} 
	\end{align} \end{subequations}
In what follows, we show that the new constraint \eqref{New-constraint} contributes in a nonredundant manner to the enriched subproblem \eqref{enriched-primal-subproblem}.\\
In view of $\bm s_{k,m+1}=\bm{\hat{x}}_{k,0}+\bm d_{k,0}=\texttt{FDP2}(\bm x_k, \bm d_{k,0}, \beta, z_{k,0}, \sigma)$, $\beta\in(0,1), z_{k,0}<0$ and \eqref{error-of-new-linearization}, one can observe that
\begin{align*}
&\bigl\langle\nabla f(\bm s_{k,m+1}),\bm d_{k,0}\bigr\rangle-e_{k,m+1}\\&
=\bigl\langle\nabla f(\bm s_{k,m+1}),\bm d_{k,0}\bigr\rangle-f(\bm x_k)+f(\bm s_{k,m+1})+ \bigl\langle\nabla f(\bm s_{k,m+1}), \bm x_k-\bm s_{k,m+1}\bigr\rangle\\&
=\bigl\langle\nabla f(\bm s_{k,m+1}),\bm d_{k,0}\bigr\rangle-f(\bm x_k)+f(\bm s_{k,m+1})+ \bigl\langle\nabla f(\bm s_{k,m+1}), \bm x_k-\hat{\bm x}_{k,0}-\bm d_{k,0}\bigr\rangle\\&
= f(\bm s_{k,m+1}) -f(\bm x_k)- \bigl\langle\nabla f(\bm s_{k,m+1}), \hat{\bm x}_{k,0}-\bm x_k\bigr\rangle \\& > \beta z_{k,0} - \lVert \nabla f(\bm s_{k,m+1})\rVert \, \lVert \hat{\bm x}_{k,0}-\bm x_k\rVert \\&>z_{k,0}-\lVert \nabla f(\bm s_{k,m+1})\rVert \lVert \hat{\bm x}_{k,0}-\bm x_k\rVert.
\end{align*}
In the last inequality,  by choosing $\sigma$ sufficiently small, we can make $ \lVert \hat{\bm x}_{k,0}-\bm x_k\rVert$ arbitrarily small, which means that
$$\bigl\langle\nabla f(\bm s_{k,m+1}),\bm d_{k,0}\bigr\rangle-e_{k,m+1}> z_{k,0}. $$
In other words, $(z_{k,0}, \bm d_{k,0})$ does not meet the new constraint \eqref{New-constraint}, and consequently, the enriched subproblem \eqref{enriched-primal-subproblem}  will generate a search direction that is different from the inefficient current direction $\bm d_{k,0}$.

Although we can now solve the dual of subproblem \eqref{enriched-primal-subproblem} to obtain a new search direction, it is wise to use the Lagrangian vector $\boldsymbol{\lambda}^{k,0}$ to drop some inefficient constraints of \eqref{Aggregated-constraints}. To this end, we use an \emph{adaptive gradient selection strategy}. Before applying this strategy, following \cite{Aggregate-sub,kiwielbook,Makela_book}, to keep the general effect of these constraints, we aggregate them into a single constraint. For this purpose, multiply each of \eqref{Aggregated-constraints} by $\lambda^{k,0}_j$ and sum up to obtain the following \emph{aggregated constraint} 
\begin{equation}\label{Singlr-Aggregated-Const}
\bigl\langle \tilde{\bm g}_{k,0}, \bm d\bigr\rangle-\tilde{e}_{k,0}\leq z.
\end{equation}
Next, we adaptively drop those constraints that their corresponding Lagrangian multipliers were small in the previous subproblem. To this end, choose  $\theta\in[0,1]$ and let $l_0$ be the largest integer that satisfies 
\begin{equation}\label{sum-of-weights}
\sum_{j=1}^{l_0}\lambda^{k,0}_{(j)}\leq \theta,  
\end{equation}
in which $\lambda^{k,0}_{(j)}$ is the $j$-th largest component of $\boldsymbol{\lambda}^{k,0}$. By dropping those constraints which are associated with Lagrangian multipliers $\lambda^{k,0}_{(j)},\, j>l_0$, the  primal subproblem \eqref{enriched-primal-subproblem} reduces to 
\begin{align*}\label{Semi-Final-Primal-Subproblem} 
\begin{split}
\min_{(z,\bm d)}  \ & z+\frac{1}{2}\varepsilon_k^{-\alpha} \lVert\bm d\rVert^2  \\
\text{s.t.} \ \ & \bigl\langle \tilde{\bm g}_{k,0}, \bm d\bigr\rangle-\tilde{e}_{k,0}\leq z, \qquad\qquad\qquad\qquad \ \ \quad\, (\text {aggregated constraint})  \\&
\bigl\langle\nabla f(\bm s_{k,(j)}),\bm d\bigr\rangle-e_{k,(j)}\leq z, \quad 1\leq j\leq l_0, \ \  (\text{effective constraints of \eqref{Aggregated-constraints}}) \\&\bigl\langle\nabla f(\bm s_{k,m+1}),\bm d\bigr\rangle-e_{k,m+1}\leq z. \qquad\qquad\quad\,(\text{new constraint \eqref{New-constraint}})
\end{split}
\end{align*} 
We note that the user-defined parameter $\theta\in[0,1]$ controls the number of selected constraints. Indeed, a larger $\theta$ keeps more constraints of \eqref{Aggregated-constraints}. In particular, since $\sum_{j=0}^{m}\lambda^{k,0}_{j}=1$, for $\theta=1$ no constraint of \eqref{Aggregated-constraints} is removed, and by choosing $\theta=0$, all of them are discarded.

In addition, to guarantee the convergence of the method, we need to have the following constraint in the above subproblem
$$\bigl\langle\nabla f(\bm s_{k,0}),\bm d\bigr\rangle-e_{k,0}\leq z,$$
which corresponds to current point $\bm x_k=\bm s_{k,0}$. Eventually, if we set
\begin{equation}\label{J_k0}
J_{k,0}:=\left\{(j) \,\,:\,\, j=1,\ldots,l_0 \right\}\cup\{0,m+1\},  
\end{equation}
we reach the following primal search direction finding subproblem at iteration $i=0$
\begin{align}\label{Final-Primal-Subproblem} 
\begin{split}
\min_{(z,\bm d)} \ \ & z+\frac{1}{2}\varepsilon_k^{-\alpha} \lVert\bm d\rVert^2 \\
\text{s.t.} \ \ & \bigl\langle \tilde{\bm g}_{k,0}, \bm d\bigr\rangle-\tilde{e}_{k,0}\leq z,  \\&
\bigl\langle\nabla f(\bm s_{k,j}),\bm d\bigr\rangle-e_{k,j}\leq z, \quad j\in J_{k,0}.
\end{split}
\end{align} 
To complete the first iteration ($i=0$), we should solve the above problem to obtain a new search direction. However, as we mentioned earlier, we prefer to solve the dual of the above problem, to obtain the new search direction $\bm d_{k,1}$. 
If $\bm d_{k,1}$ satisfies a sufficient decrease condition, then the improvement process is completed. 
Otherwise, we need to perform more iterations of this process.

In what follows, we discuss the $i$-th iteration ($i\ge1$) of the improvement process. We note that at the beginning of this iteration, the information 
$\bm d_{k,i}, \boldsymbol{\lambda}^{k,i}, \tilde{\bm g}_{k,i}, w_{k,i}, v_{k,i},\tilde{e}_{k,i}$ and  $J_{k,i-1}$ are in hand from the $(i-1)$-th iteration.
At first, we set
$$\bm s_{k,m+i+1}:= \texttt{FDP2}(\bm x_k, \bm d_{k,i}, \beta, z_{k,i}, \sigma),$$ 
and we consider the corresponding linearization
\begin{equation}
f_{k,m+i+1}(\bm x):= f(\bm s_{k,m+i+1})+ \bigl\langle\nabla f(\bm s_{k,m+i+1}),x-\bm s_{k,m+i+1}\bigr\rangle ,
\end{equation} 
with the linearization error
\begin{equation}
e_{k,m+i+1}:=f(\bm x_k)-f_{k,m+i+1}(\bm x_k).
\end{equation}
Moreover, suppose that the following criterion holds for this new linearization
\begin{equation}\label{Criterion-i}
e_{k,m+i+1}\leq \gamma \tilde{e}_{k,i} \quad \text{or} \quad \lvert f(\bm x_k+\bm d_{k,i})- f(\bm x_k)\rvert \leq \frac{1}{2}\lVert \tilde{\bm g}_{k,i}\lVert^2+\tilde{e}_{k,i}.
\end{equation}
Now similar to the preceding  discussion, having aggregated all constraints of the previous subproblem, we discard some extra indices of $J_{k,i-1}$ to define 
\begin{equation}\label{J_ki}
J_{k,i}=\{(j)\,\, : \,\, j=1,\ldots,l_{i} \,\, \text{and} \,\, (j)\in J_{k,i-1}\}\cup \{0,m+i+1\},
\end{equation}
in which $l_{i}$ is the largest integer that satisfies
\[
\frac{\sum_{j=1}^{l_{i}}\lambda^{k,i}_{(j)}}{\sum_{j\in J_{k,i-1}\lambda^{k,i}_j}}\leq \theta, \quad \text{for some} \,\, \theta\in[0,1].
\]

After computing $J_{k,i}$, we are in a position to consider the following primal search direction finding subproblem at iteration $i\geq 0$
\begin{align}\label{FFinal-Primal-Subproblem}
\begin{split}
\min_{(z,\bm d)} \ \ & z+\frac{1}{2}\varepsilon_k^{-\alpha} \lVert\bm d\rVert^2 \\ \text{s.t.}\ \ & 
\bigl\langle \tilde{\bm g}_{k,i} ,\bm d\bigr\rangle-\tilde e_{k,i}\leq z,\\& \bigl\langle\nabla f(\bm s_{k,j}),\bm d\bigr\rangle-e_{k,j}\leq z, \quad j\in J_{k,i}.
\end{split}
\end{align}
Note that, if $(z_{k,i+1},\bm d_{k,i+1})$ solves subproblem \eqref{FFinal-Primal-Subproblem}, then 
\begin{equation}\label{z-k-i f-tilde-k-p}
z_{k,i+1}=\max\left\{\bigl\langle \tilde{\bm g}_{k,i} ,\bm d_{k,i+1}\bigr\rangle-\tilde e_{k,i},\, \bigl\langle\nabla f(\bm s_{k,j}),\bm d_{k,i+1}\bigr\rangle-e_{k,j}\,;\, j\in J_{k,i}  \right\}.
\end{equation}
Finally, we consider the dual of subproblem \eqref{FFinal-Primal-Subproblem} which is given by
\begin{align}\label{Dual-sub-in process}
&\min_{\boldsymbol{\lambda}} \, \frac{1}{2} \lVert \sum_{j\in J_{k,i}} \lambda_j \nabla f(\bm s_{k,j})+\tilde{\lambda} \,\tilde{\bm g}_{k,i}\rVert^2 +\frac{1}{\varepsilon_k^\alpha} (\sum_{j\in J_{k,i}} \lambda_j e_{k,j}+\tilde{\lambda} \,\tilde{e}_{k,i}) \nonumber \\  & {\rm{s.t.}} \quad \sum_{j\in J_{k,i}}\lambda_j+\tilde{\lambda} =1,\quad \tilde{\lambda}\geq 0, \quad \lambda_j\geq 0, \quad j\in J_{k,i}.
\end{align} 
Suppose that $\boldsymbol{\lambda}^{k,i+1}=\left(\lambda^{k,i+1}_j;j\in J_{k,i}, \tilde{\lambda}^{k,i+1}\right)$ solves subproblem \eqref{Dual-sub-in process}. Then, the following updates are performed
\begin{align}
&\tilde{\bm g}_{k,i+1}:=\sum_{j\in J_{k,i}} \lambda^{k,i+1}_j \nabla f(\bm s_{k,j})+\tilde{\lambda}^{k,i+1}\, \tilde{\bm g}_{k,i},\label{Update-p1}\\&
\tilde{e}_{k,i+1}:=\sum_{j\in J_{k,i}} \lambda^{k,i+1}_j e_{k,j}+\tilde{\lambda}^{k,i+1}\,\tilde{e}_{k,i},\label{Update-p2}
\\
&z_{k,i+1}:=-\varepsilon_k^\alpha \lVert \, \tilde{\bm g}_{k,i+1}\rVert^2- \tilde{e}_{k,i+1},\label{Update-p3a} \\
&\bm d_{k,i+1}:=-\varepsilon_k^\alpha \, \tilde{\bm g}_{k,i+1},\label{Update-p3b}\\
& v_{k,i+1}:= \frac{1}{2}\lVert \tilde{\bm g}_{k,i+1}\rVert^2 + \tilde{e}_{k,i+1}, \label{Update-p4}\\
&w_{k,i+1}:= \frac{1}{2} \lVert \tilde{\bm g}_{k,i+1}\rVert^2+\frac{1}{\varepsilon^{\alpha}_k} \, \tilde{e}_{k,i+1}. \label{optimal-value-of-Dualsub-in-process} 
\end{align} 
In the case that $\bm d_{k,i+1}$ satisfies the following sufficient decrease condition
\begin{equation}\label{Suf-condition-in process}
f(\bm x_k+ \bm d_{k,i+1})-f(\bm x_k)\leq \beta z_{k,i+i} \,,
\end{equation}
we set
$$\bm x_{k+1}:= \texttt{FDP1}(\bm x_k, \bm d_{k,i+1}, \beta, z_{k,i+1},\sigma), $$
and we terminate the improvement process to repeat the procedure of subsection \ref{subsec-3.1} for the new trial point $\bm x_{k+1}$. Otherwise, we increase $i$ by one and the improvement process is repeated. 

In what follows, we show that the variable $v_{k,i}$ still measures the stationarity of the current point $\bm x_k$. For this purpose, we recall from \eqref{Sub-inequality-for-tilde-g} that
\begin{equation}\label{c1}
f(\bm x)\geq f(\bm x_k)+ \bigl\langle \tilde{\bm g}_{k,0}, \bm x-\bm x_k\bigr\rangle -\tilde{e}_{k,0}, \qquad \text{for all} \,\, \bm x\in\mathbb{R}^n.
\end{equation} 
Furthermore, similar to \eqref{e-nabla-inequality}, we know that  
\begin{equation}\label{c2}
f(\bm x)\geq f(\bm x_k)+ \bigl\langle \nabla f(\bm s_{k,j}), \bm x-\bm x_k\bigr\rangle -e_{k,j}, \  \text{for all} \,\, \bm x\in\mathbb{R}^n \,\, \text{and} \,\, j\in J_{k,0}.
\end{equation}
Now, by multiplying the  inequality \eqref{c1}  by $\tilde{\lambda}^{k,1}$ and each of inequalities \eqref{c2} by $\lambda^{k,1}_j$ and summing up, we obtain (use also \eqref{Update-p1} and \eqref{Update-p2})
\begin{equation*}
f(\bm x)\geq f(\bm x_k)+ \bigl\langle \tilde{\bm g}_{k,1}, \bm x-\bm x_k\bigr\rangle -\tilde{e}_{k,1}, \quad {\rm for \,\, all} \,\, \bm x\in\mathbb{R}^n,
\end{equation*}
which means that $\tilde{\bm g}_{k,1}\in \partial_{\tilde{e}_{k,1}}f(\bm x_k)$. By repeating this process inductively, we observe that
\begin{equation}\label{Sub-inequality-in-process}
f(\bm x)\geq f(\bm x_k)+ \bigl\langle \tilde{\bm g}_{k,i}, \bm x-\bm x_k\bigr\rangle -\tilde{e}_{k,i}, \quad {\rm for \,\, all} \,\, \bm x\in\mathbb{R}^n  \,\, \text{and}\,\, i,
\end{equation}
and hence
\begin{equation}\label{tilde-e-g-in-process}
\tilde{\bm g}_{k,i}\in \partial_{\tilde{e}_{k,i}}f(\bm x_k), \quad \text{for\,all}\,\, i.
\end{equation}
Thus, the variable 
$$v_{k,i}=\frac{1}{2}\lVert \tilde{\bm g}_{k,i}\rVert^2 + \tilde{e}_{k,i},$$ 
measures the stationarity of the current point $\bm x_k$. In this respect, if the value of this variable is less than some optimality tolerance, we terminate the whole algorithm.  

\subsection{Bundle-GS algorithm}\label{Bundle-GS Algorithm}
Here, based on the previous discussions, a version of bundle methods,
namely the Bundle-GS (B-GS) method is presented  in Algorithm \ref{B-GS-Alg}.
In this algorithm, the outer loop uses the GS technique to build up an initial model for the objective function $f$, while the inner loop checks the efficiency of this search direction and, if necessary, it runs the improvement process.
Also, the functions \texttt{FDP1} and \texttt{FDP2} play their roles to keep the new iterations and auxiliary points in $D$. The B-GS algorithm is terminated whenever the value of $v_{k,i}$ is less than some optimality tolerance.

\LinesNumbered

\IncMargin{1em}
\begin{algorithm}[]
	%\SetAlgoLined
	\SetAlgoNoLine
	\SetKwInOut{Input}{inputs}
	\SetKwInOut{Output}{output}
	\SetKwInOut{Required}{requied}
	\SetKwFor{Loop}{Loop}{}{End\,Loop}
	
	\Indm  
	\Input{ \rule{0pt}{4ex}  $\bm x_0\in D$ as starting point, initial sampling radius $\varepsilon_0>0$ together with its reduction factor $\mu\in(0,1)$, sample size $m\in\mathbb{N}_0$, sufficient decrease parameter $\beta\in(0,1)$, maximum allowable perturbation $\sigma>0$, scale parameters $\alpha,\gamma>0$, maximum weight $\theta
		\in[0,1]$, and termination tolerance $\epsilon_{tol}>0$. }
	
	\Indp
	\BlankLine
	
	Set $k:=0$;
	
	\Loop{}{
		Sample $\mathbf s_{k,1},\ldots,\mathbf s_{k,m}$  independently and uniformly from $B(\bm x_k,\varepsilon_k)$ and set $\bm s_{k,0}:=\bm x_k$\;
		\If{for some $j=1,\ldots,m$ the point $\bm s_{k,j}\notin D$}{\textbf{STOP}!\;}
		Solve subproblem \eqref{Dual-sub} to compute $\boldsymbol{\lambda}^{k,0}$\;
		Compute $ \tilde{\bm g}_{k,0}, \tilde{e}_{k,0}\, z_{k,0}, \bm d_{k,0},  \, \text{and}\,\, v_{k,0}$  by \eqref{tilde-g,e-0}, \eqref{Rep-d-and-z-by-aggregate},  and \eqref{Opt-Tol} respectively\;
		Form $J_{k,0}$ as described in \eqref{J_k0}\;
		Set  $i:=0$\;
		\Loop{}{
			
			\If{$v_{k,i}\leq\epsilon_{\rm tol}$}
			{
				Return ${\mathbf x}_k$ as an approximation of the minimum point and \textbf{STOP}\;
			}
			\eIf{$f(\bm x_k+\bm d_{k,i})-f(\bm x_k)\leq \beta\, z_{k,i}$ }{
				Set $\bm x_{k+1}:=\texttt{FDP1}(\bm x_k,\bm d_{k,i},\beta,z_{k,i},\sigma)$\, and\, $\varepsilon_{k+1}:=\varepsilon_k$\;
				\textbf{break}\;
			}
			{
				Set $\bm s_{k,m+i+1}:=\texttt{FDP2}(\bm x_k, \bm d_{k,i},\beta, z_{k,i},\sigma)$\;
				\eIf{criterion \eqref{Criterion-i} holds}{
					Solve subproblem \eqref{Dual-sub-in process} to compute $\boldsymbol{\lambda}^{k,i+1}$\;
					Compute $\tilde{\bm g}_{k,i+1}, \tilde{e}_{k,i+1}$, $z_{k,i+1}$, $\bm d_{k,i+1}$, and $v_{k,i+1}$  by \eqref{Update-p1}-\eqref{Update-p4}, respectively\; 	
					Form $J_{k,i+1}$ as described in \eqref{J_ki}\;
				}
				{
					Set $\varepsilon_{k+1}:=\mu \varepsilon_k$ and  $\bm x_{k+1}:=\bm x_k$\;
					\textbf{break}\;
				}
			}
			Set $i:=i+1$\;
		}
		%end of i loop 
		Set $k=:k+1$\;
	}%end of k loop	

	\caption{ \rule{0pt}{2.5ex}Bundle-GS (B-GS) Algorithm }\label{B-GS-Alg}
	\label{BGS}
\end{algorithm}
\DecMargin{1em}

\section{Convergence analysis}\label{Convergence analysis}
In this section, we study the convergence of the presented B-GS algorithm. In fact, we shall show that, whenever $f$ attains its minimum, each sequence $\{x_k\}$ generated by Algorithm \ref{B-GS-Alg} converges to a minimum point. Our convergence analysis closely follows the works of Kiwiel in \cite{Aggregate-sub,kiwielbook} and Schramm and Zowe in \cite{BT-method}. Throughout this subsection, we assume that the optimality tolerance $\epsilon_{\text{tol}}$ is set to zero. This allows Algorithm \ref{B-GS-Alg} to generate an infinite sequence of iterations, and consequently, we can provide some asymptotic convergence results.
We start with the following theorem which studies the convergence of the method in the case of finite number of iterations. 
\begin{theorem}
	Suppose that Algorithm \ref{B-GS-Alg} terminates at the $k$-th iteration. Then, with probability 1, $\bm x_k$ is a minimum point for the objective function $f$.	
\end{theorem}
\begin{proof}
	From measure theory, we know that the termination in Line $5$ has a zero probability. In addition, the functions \texttt{FDP1} and \texttt{FDP2} are well defined, with probability 1. Therefore, if the algorithm terminates at the $k$-th iteration, it did so in Line $13$, with probability 1. Thus, we restrict our attention to the case that there exists $i_k\in\mathbb{N}_0$ such that
	$$v_{k,i_k}=\frac{1}{2}\lVert \tilde{\bm g}_{k,i_k}\rVert^2+\tilde{e}_{k,i_k} =0.$$
	This gives that $\tilde{\bm g}_{k,i_k}=\bm 0$ and $\tilde{e}_{k,i_k}=0$. Now a glance at \eqref{tilde-e-g-in-process} reveals that  $\bm 0\in\partial f(\bm x_k)$.   \qed
\end{proof}

From now on, we suppose that the algorithm does not terminate by the stopping criterion in Line $12$. In other words, we suppose $v_{k,i}>0$ for every $k$ and $i$. Then, with probability 1, two cases may occur:\\

\emph{Case I.} Algorithm \ref{B-GS-Alg} generates an infinite sequence $\{\bm x_k\}_{k\in\mathbb{N}_0}$. In this case, for each $k\in\mathbb{N}_0$ the inner loop of this algorithm terminates after finitely many iterations.\\
\emph{Case II.} For a $\bar{k}\in\mathbb{N}_0$ the inner loop of Algorithm \ref{B-GS-Alg} does not terminate. Indeed, at the $\bar{k}$-th iteration of the algorithm, we have $i\to\infty$.\\\\
The convergence of the method is separately studied for each case. First, we consider \emph{Case I}, in which Algorithm \ref{B-GS-Alg} produces an infinite sequence $\{\bm x_k\}_{k\in\mathbb{N}_0}$.
Throughout the study of \emph{Case I}, we assume that the following assumption holds \cite{kiwielbook,Aggregate-sub,BT-method}.

\begin{assumption}\label{Assumption}
	There is a point $\bar{\bm x}\in\mathbb{R}^n$ that satisfies $f(\bar{\bm x})\leq f(\bm x_k)$ for every $k\in\mathbb{N}_0$.
\end{assumption}

First of all, we need to show that the sequence $\{\bm x_k\}_{k\in\mathbb{N}_0}$ is a convergent sequence. To this end, we start with the following auxiliary result.
\begin{lemma}\label{Lemma-1}
	Suppose that Algorithm \ref{B-GS-Alg} generates the infinite sequence $\{\bm x_k\}_{k\in\mathbb{N}_0}$ and Assumption \ref{Assumption} holds. Then for each $\delta>0$, there exists $n_{\delta}\in\mathbb{N}_0$ such that
	$$\lVert \bar{\bm x}-\bm x_{k+1}\rVert^2\leq\lVert \bar{\bm x}-\bm x_{k'}\rVert^2 + \delta, \quad {\rm{for}}\,\,\, k\geq k'\geq n_\delta. $$	
\end{lemma}
\begin{proof}
	Let $\mathcal{K}$ be a subset of $\mathbb{N}_0$ such that for every $k\in\mathcal{K}$, the point $\bm x_{k}$ is updated by Line $16$ of Algorithm \ref{B-GS-Alg}. In other words, for each $k\in\mathcal{K}$, there is $i_k\in\mathbb{N}_0$ such that
	$$\bm x_{k+1}=\texttt{FDP1}(\bm x_k, \bm d_{k,i_k}, \beta, z_{k,i_k}, \sigma). $$
	Consequently, by construction, for every $k\in\mathbb{N}_0\setminus\mathcal{K}$, we have $\bm x_{k+1}=\bm x_k$. For each $k\in\mathcal{K}$, we can write
	\begin{align}
	\bm x_{k+1}- \bm x_{k}&= \texttt{FDP1}(\bm x_k, \bm d_{k,i_k}, \beta, z_{k,i_k}, \sigma)- \bm x_k \nonumber \\& 
	=(\bm{\hat{x}}_{k,i_k}+\bm d_{k,i_k})-\bm x_k \label{FDP-rep} \\& = (\bm{\hat{x}}_{k,i_k}-\bm x_k) -\varepsilon^{\alpha}_k \, \tilde{\bm g}_{k,i_k} \quad ({\rm use} \eqref{Update-p3b}) . \label{xk+1-xk-rep}
	\end{align} 
	In addition, we know from \eqref{Sub-inequality-in-process} that
	$$\bigl\langle \tilde{\bm g}_{k,i_k}, \bm x-\bm x_k\bigr\rangle \leq f(\bm x)-f(\bm x_k)+\tilde{e}_{k,i_k}, \quad {\rm for \,\, all} \,\, \bm x\in\mathbb{R}^n. $$
	Since $f(\bm{\bar{x}})\leq f(\bm x_k)$ for all $k$, setting $\bm x=\bm{\bar{x}}$ in the above inequality, we obtain
	\begin{equation}\label{bar x-tilde-g-inequality}
	\ \bigl\langle \tilde{\bm g}_{k,i_k}, \bm {\bar x}-\bm x_k\bigr\rangle \leq \tilde{e}_{k,i_k}.
	\end{equation}
	Now, using \eqref{xk+1-xk-rep} and \eqref{bar x-tilde-g-inequality}, for every $k\in\mathcal{K}$, we can write
	\begin{align*}%\label{CS-inequality}
	\lVert \bar{\bm x}-\bm x_{k+1}\rVert^2&=\lVert \bar{\bm x}-\bm x_{k}\rVert^2+ \lVert \bm x_k-\bm x_{k+1}\rVert^2-2\ \bigl\langle \bar{\bm x}-\bm x_k, \bm x_{k+1}-\bm x_k\bigr\rangle \nonumber \\&
	=\lVert \bar{\bm x}-\bm x_{k}\rVert^2+ \lVert \bm x_k-\bm x_{k+1}\rVert^2-2\,\ \bigl\langle \bar{\bm x}-\bm x_k, (\hat{\bm x}_{k,i_k}-\bm x_k) -\varepsilon^{\alpha}_k \tilde{\bm g}_{k,i_k}\bigr\rangle \nonumber \\& 
	\leq \lVert \bar{\bm x}-\bm x_{k}\rVert^2+ \lVert \bm x_k-\bm x_{k+1}\rVert^2 +2\, \varepsilon^{\alpha}_k \tilde{e}_{k,i_k}+ 2 \ \bigl\langle \bm x_k-\bar{\bm x}, \hat{\bm x}_{k,i_k}-\bm x_k\bigr\rangle \nonumber \\&
	\leq \lVert \bar{\bm x}-\bm x_{k}\rVert^2+ \lVert \bm x_k-\bm x_{k+1}\rVert^2 +2\, \varepsilon^{\alpha}_k \tilde{e}_{k,i_k}+ 2 \lVert \bm x_k-\bar{\bm x} \rVert \, \lVert \hat{\bm x}_{k,i_k}-\bm x_k \rVert. 
	\end{align*}
	Note that, in the last inequality we can make the term $\lVert \hat{\bm x}_{k,i_k}-\bm x_k \rVert$ as small as we like by choosing $\sigma>0$ sufficiently small. Thus,  the above inequality implies that   
	\begin{equation}\label{Main-1}
	\lVert \bar{\bm x}-\bm x_{k+1}\rVert^2 \leq \lVert \bar{\bm x}-\bm x_{k}\rVert^2+ \lVert \bm x_k-\bm x_{k+1}\rVert^2 +2\, \varepsilon^{\alpha}_k \tilde{e}_{k,i_k}, \quad \text{for \,all} \,\, k\in\mathcal{K}.
	\end{equation} 
	Furthermore, since $\bm x_{k+1}=\bm x_k$ for every $k\in\mathbb{N}_0\setminus\mathcal{K}$, we have
	\begin{equation}\label{Main-2}
	\lVert \bar{\bm x}-\bm x_{k+1}\rVert^2=\lVert \bar{\bm x}-\bm x_k\rVert^2, \quad \text{for \,all} \,\, k\in\mathbb{N}_0\setminus\mathcal{K}.
	\end{equation}
	Now using \eqref{Main-1} and \eqref{Main-2} inductively, for every $k,k'\in\mathbb{N}_0$ and $k\geq k'$, we obtain
	\begin{equation}\label{Main-3}
	\lVert \bar{\bm x}-\bm x_{k+1}\rVert^2\leq\lVert \bar{\bm x}-\bm x_{k'}\rVert^2+\sum_{\substack{j=k'\\j\in\mathcal{K}}}^{k}\left( \lVert \bm x_j-\bm x_{j+1}\rVert^2 +2 \, \varepsilon^{\alpha}_j \tilde{e}_{j,i_j} \right).
	\end{equation}
	Next, we consider the sum in \eqref{Main-3}. First, we note that for every $n\geq 1$, we can write
	\begin{align}
	f(\bm x_0)-f(\bm x_n)&=\left[f(\bm x_0)-f(\bm x_1)\right]+\ldots+\left[f(\bm x_{n-1})-f(\bm x_{n})\right] \nonumber \\&
	\geq -\beta \sum_{\substack{j=1\\j\in\mathcal{K}}}^{n} z_{j,i_j} = \beta \sum_{\substack{j=1\\j\in\mathcal{K}}}^{n} \left(\varepsilon^{\alpha}_j \, \lVert \tilde{\bm g}_{j,i_j} \rVert^2+\tilde{e}_{j,i_j} \right)\quad {\rm{(by \,\, \eqref{Update-p3a})   }}  \nonumber \\&
	=\beta \sum_{\substack{j=1\\j\in\mathcal{K}}}^{n} \left( \frac{1}{\varepsilon^{\alpha}_j}\, \lVert \bm d_{j,i_j}\rVert^2 +\tilde{e}_{j,i_j}  \right) \quad \rm{(by \,\, \eqref{Update-p3b})}. \label{Serirs-inequality} 
	\end{align}
	Moreover, in view of \eqref{FDP-rep}, we have
	$$\lVert \bm d_{j,i_j}\rVert=\lVert (\bm x_{j+1}-\bm x_j)+(\bm x_j-\hat{\bm x}_{j,i_j})\rVert\geq \lVert \bm x_{j+1}-\bm x_j\rVert - \lVert \bm x_j-\hat{\bm x}_{j,i_j} \rVert, $$
	and again we can make $\lVert \bm x_j-\hat{\bm x}_{j,i_j} \rVert$ as small as we like by choosing the parameter $\sigma$ sufficiently small, which means that
	$$\lVert \bm d_{j,i_j}\rVert \geq \lVert \bm x_{j+1}-\bm x_j\rVert.  $$
	Therefore, we can continue \eqref{Serirs-inequality} with
	\begin{equation}
	f(\bm x_0)-f(\bm x_n)\geq \beta \sum_{\substack{j=1\\j\in\mathcal{K}}}^{n} \left( \frac{1}{\varepsilon^{\alpha}_j} \lVert \bm x_{j+1}-\bm x_j\rVert^2 +\tilde{e}_{j,i_j} \right).
	\end{equation}
	Letting $n$ approach infinity and using the fact that $f(\bar{x})\leq f(\bm x_k)$ for every $k$, we conclude
	\begin{align}
	\infty&>f(\bm x_0)-f(\bm{\bar x})\geq \beta \sum_{\substack{j=1\\j\in\mathcal{K}}}^{\infty} \left( \frac{1}{\varepsilon^{\alpha}_j} \lVert \bm x_{j+1}-\bm x_j\rVert^2 +\tilde{e}_{j,i_j} \right) \nonumber \\ &
	\geq \frac{\beta} {\varepsilon^{\alpha}_0} \sum_{\substack{j=1\\j\in\mathcal{K}}}^{\infty} \left(\lVert \bm x_{j+1}-\bm x_j\rVert^2 + \varepsilon^{\alpha}_j\,\tilde{e}_{j,i_j} \right) \quad ({\rm by} \,\, \frac{\varepsilon^{\alpha}_j} {\varepsilon^\alpha_0}\leq1) \nonumber \\&
	\geq \frac{\beta} {2\varepsilon^{\alpha}_0} \sum_{\substack{j=1\\j\in\mathcal{K}}}^{\infty} \left(\lVert \bm x_{j+1}-\bm x_j\rVert^2 + 2\,\varepsilon^{\alpha}_j\,\tilde{e}_{j,i_j} \right).
	\end{align}
	This means that, we can make the sum in \eqref{Main-3} smaller than $\delta$ by choosing $k'$ sufficiently large. This completes the proof. \qed	
\end{proof} 
In the next Lemma, we show that the sequence $\{\bm x_k\}_{k\in\mathbb{N}_0}$ is indeed a convergent sequence.

\begin{corollary}\label{Corollary1}
	Under the assumptions of Lemma \eqref{Lemma-1}, the sequence $\{\bm x_k\}_{k\in\mathbb{N}_0}$ converges to some $\bm{\tilde{x}}\in\mathbb{R}^n$ such that
	$$f(\tilde{\bm x})\leq f(\bm x_k)\quad {\rm{for\, all \,\, }} k. $$ 
\end{corollary}
\begin{proof}
	By Lemma \ref{Lemma-1} we know that the sequence $\{\bm x_k\}_{k\in\mathbb{N}_0}$ is bounded, and therefore, it has at least one accumulation point, say $\bm{\tilde{x}}$. We know, by construction, the sequence $\{f(\bm x_k)\}_{k\in\mathbb{N}_0}$ is a nonincreasing sequence, which gives 
	$$f(\tilde{\bm x})\leq f(\bm x_k)\quad {\rm{for\, all \,\, }} k. $$
	Now, let $\epsilon>0$ be arbitrary and apply Lemma \ref{Lemma-1} for $\bm{\tilde x}$ and $\delta=\epsilon/2$. Then there exists $n_{\epsilon/2}\in\mathbb{N}_0$ such that
	\begin{equation}\label{Lemma2-1}
	\lVert \tilde{\bm x}-\bm x_{k+1}\rVert^2\leq\lVert \tilde{\bm x}-\bm x_{k'}\rVert^2 + \epsilon/2, \quad {\rm{for}}\,\,\, k\geq k'\geq n_{\epsilon/2}.
	\end{equation}
	Furthermore, since $\bm{\tilde{x}}$ is an accumulation point of $\{\bm x_k\}_{k\in\mathbb{N}_0}$, there is a $\bar k\in\mathbb{N}_0$ and $\bar k\geq n_{\epsilon/2}$ such that
	\begin{equation}\label{Lemma2-2}
	\lVert \tilde{\bm x}-\bm x_{\bar k}\rVert^2\leq \epsilon/2.
	\end{equation}
	Combining \eqref{Lemma2-1} and \eqref{Lemma2-2}, we see
	$$\lVert \tilde{\bm x}-\bm x_{k+1}\rVert^2\leq \lVert \tilde{\bm x}-\bm x_{\bar k}\rVert^2+ \epsilon/2\leq \epsilon/2+\epsilon/2=\epsilon, \quad {\rm{for\, all}}\,\, k\geq \bar k.$$
	Since $\epsilon>0$ was arbitrary, this proves $\bm x_k \to \bm{\tilde{x}}$ as $k\to\infty$. \qed
\end{proof}	
It remains to show that $\tilde{\bm x}$ is indeed optimal for $f$. Before it, we need to establish the following lemma and its conclusion.
\begin{lemma}\label{Lemma2}
	Suppose that Algorithm \ref{B-GS-Alg} generates the infinite sequence $\{\bm x_k\}_{k\in\mathbb{N}_0}$ and Assumption \ref{Assumption} holds. Let
	$$\mathcal{A}:=\{(k,i)\in\mathbb{N}_0\times\mathbb{N}_0\,\,:\,\, \text{Algorithm} \ \ref{B-GS-Alg} \,\, \text{generates subscripts}\,\, (k,i)\}.$$ 
	Then the sequence $\{w_{k,i}\}_{(k,i)\in\mathcal{A}}$ is a bounded sequence.
\end{lemma}
\begin{proof}
	We know by construction that $(k,0)\in\mathcal{A}$, for every $k\in\mathbb{N}_0$. Now, we have 
	\begin{align*}
	w_{k,0}&=w_k=\frac{1}{2}\lVert \tilde{\bm g}_k\rVert^2+\frac{1}{\varepsilon^\alpha_k}\tilde{e}_k\\&
	=-\frac{1}{\varepsilon^\alpha_k}(z_k+\frac{1}{2\varepsilon^\alpha_k}\lVert\bm d_k\rVert^2) \quad (\text{use}\,\, \eqref{Rep-d-and-z-by-aggregate}  ) \\&
	=-\frac{1}{\varepsilon^\alpha_k}\left( \max_{j=0,\ldots,m}\left\{\ \bigl\langle \nabla f(\bm s_{k,j}), \bm d_k\bigr\rangle-e_{k,j}\right\} +\frac{1}{2\varepsilon^\alpha_k} \lVert\bm d_k\rVert^2 \right) \quad (\text{use}\,\, \eqref{z andf-k-p}) \\&
	\leq -\frac{1}{\varepsilon^\alpha_k} \left( \bigl\langle \nabla f(\bm s_{k,0}), \bm d_k\bigr\rangle + \frac{1}{2\varepsilon^\alpha_k} \lVert\bm d_k\rVert^2  \right)\quad (\,\text{by}\,\, e_{k,0}=0)\\&
	\leq -\frac{1}{\varepsilon^\alpha_k} \min_{\bm d\in\mathbb{R}^n}\left\{ \bigl\langle \nabla f(\bm s_{k,0}), \bm d\bigr\rangle + \frac{1}{2\varepsilon^\alpha_k} \lVert\bm d\rVert^2 \right\} \\&
	= \frac{1}{2} \lVert \nabla f(\bm s_{k,0})\rVert^2 \quad \left(\,\text{use the minimizer}\,\,\bm d^*=-\varepsilon^\alpha_k\,\nabla f(\bm s_{k,0})\right)\\&
	= \frac{1}{2} \lVert \nabla f(\bm x_k)\rVert^2 \quad (\,\text{by}\,\, \bm s_{k,0}=\bm x_k ).
	\end{align*}
	Therefore
	\begin{equation}\label{A1}
	w_{k,0}\leq \frac{1}{2} \lVert \nabla f(\bm x_k)\rVert^2, \quad \text{for all}\,\, (k,0)\in\mathcal{A}.
	\end{equation}
	In a similar fashion, for every $(k,i)\in\mathcal{A}$ and $i\geq1$, we can write
	\begin{align*}
	w_{k,i}&=\frac{1}{2}\lVert \tilde{\bm g}_{k,i}\rVert^2+\frac{1}{\varepsilon^\alpha_k}\tilde{e}_{k,i}\\&
	=-\frac{1}{\varepsilon^\alpha_k}(z_{k,i}+\frac{1}{2\varepsilon^\alpha_k}\lVert\bm d_{k,i}\rVert^2) \quad (\text{use}\,\, \eqref{Update-p3a}\,\, {\rm and} \,\, \eqref{Update-p3b} ) \\&
	=-\frac{1}{\varepsilon^\alpha_k}\left( \max_{j\in J_{k,i-1}}\left\{\bigl\langle \tilde{\bm g}_{k,i} ,\bm d_{k,i}\bigr\rangle-\tilde e_{k,i},\, \bigl\langle\nabla f(\bm s_{k,j}),\bm d_{k,i}\bigr\rangle-e_{k,j} \right\}\right. \\& \left. \qquad\qquad +\frac{1}{2\varepsilon^\alpha_k} \lVert\bm d_{k,i}\rVert^2 \right) \quad (\text{use}\,\, \eqref{z-k-i f-tilde-k-p}) \\&
	\leq -\frac{1}{\varepsilon^\alpha_k} \left( \bigl\langle \nabla f(\bm s_{k,0}), \bm d_{k,i}\bigr\rangle + \frac{1}{2\varepsilon^\alpha_k} \lVert\bm d_{k,i}\rVert^2  \right)\quad (\,\text{by}\,\, 0\in J_{k,i-1} \,\, \text{and}\,\, e_{k,0}=0)\\&
	\leq -\frac{1}{\varepsilon^\alpha_k} \min_{\bm d\in\mathbb{R}^n}\left\{ \bigl\langle \nabla f(\bm s_{k,0}), \bm d\bigr\rangle + \frac{1}{2\varepsilon^\alpha_k} \lVert\bm d\rVert^2 \right\} \\&
	= \frac{1}{2} \lVert \nabla f(\bm s_{k,0})\rVert^2 = \frac{1}{2} \lVert \nabla f(\bm x_k)\rVert^2 .
	\end{align*}
	Hence
	\begin{equation}\label{A2}
	w_{k,i}\leq \frac{1}{2} \lVert \nabla f(\bm x_k)\rVert^2, \quad \text{for all}\,\, (k,i)\in\mathcal{A}\,\, \text{and}\,\, i\geq1.
	\end{equation}
	Combining \eqref{A1} and \eqref{A2}, we can write
	\begin{equation*}
	w_{k,i}\leq \frac{1}{2} \lVert \nabla f(\bm x_k)\rVert^2, \quad \text{for all}\,\, (k,i)\in\mathcal{A}.
	\end{equation*}
	By Lemma \ref{Corollary1}, we know that $\{\bm x_k\}_{k\in\mathbb{N}_0}$ is a convergent sequence. Therefore, the local boundedness of the map $\nabla f:\mathbb{R}^n\to\mathbb{R}^n$ implies the existence of $C>0$ such that
	$$\lVert \nabla f(\bm x_k)\rVert^2\leq 2C, \quad \text{for all}\,\, k\in\mathbb{N}_0,  $$
	and consequently
	$$w_{k,i}\leq C,\quad \text{for all}\,\, (k,i)\in\mathcal{A}. $$ \qed
\end{proof}

\begin{corollary}\label{Corollary2}
	Under the assumptions of Lemma \ref{Lemma2}, the sequences $\{ \bm d_{k,i} \}_{(k,i)\in\mathcal{A}}$ and $\{ \bm{\tilde{g}}_{k,i} \}_{(k,i)\in\mathcal{A}}$ are bounded.
\end{corollary}
\begin{proof}
	We know that
	$$w_{k,i}=\frac{1}{2} \lVert \bm{\tilde{g}}_{k,i}\rVert^2+\tilde{e}_{k,i}, \quad \text{for all} \,\, (k,i)\in\mathcal{A}. $$
	Thus, the boundedness of $\{w_{k,i}\}_{(k,i)\in\mathcal{A}}	$ yields the boundedness of the sequence $\{ \bm{\tilde{g}}_{k,i}\}_{(k,i)\in\mathcal{A}}$. Now, since $0<\varepsilon^\alpha_k\leq \varepsilon^\alpha_0 $ for every $k$, the boundedness of $\{ \bm d_{k,i}\}_{ (k,i)\in\mathcal{A}}$ follows from 
	$$\bm d_{k,i}=-\varepsilon^{\alpha}_k\, \bm{\tilde{g}}_{k,i}, \quad \text{for all} \,\, (k,i)\in\mathcal{A}.  $$ \qed
\end{proof}
Now we are in a position to state the principal result. Before it, let us denote the set of minimum points by
\[X^*:=\{\bm x^*\in\mathbb{R}^n\,\, : \,\, f(\bm x^*)\leq f(\bm x) \,\,\, \text{for all} \,\, \bm x\in\mathbb{R}^n \}. \]
\begin{theorem}\label{Th1}
	Suppose that Algorithm \ref{B-GS-Alg} generates the infinite sequence $\{\bm x_k\}_{k\in\mathbb{N}_0}$ and $X^*\neq\emptyset$. Then $\bm x_k\to \bm{\tilde{x}}$ such that
	$$\bm 0\in \partial f(\bm{\tilde{x}}).$$			
\end{theorem}
\begin{proof}
	Since $X^*\neq\emptyset$, obviously Assumption \ref{Assumption} holds. Thus, by Corollary \ref{Corollary1}, we know that $\bm x_k\to \bm{\tilde{x}}$. To show $\bm 0\in \partial f(\bm{\tilde{x}})$,  we need to consider two cases.\\
	\emph{Case a.} $\varepsilon^\alpha_k\downarrow0$ as $k\to\infty$. 
	
	First, we show that there is $\mathcal{A}'\subset\mathcal{A}$ such that
	$$v_{k,i}=\frac{1}{2}\lVert \bm{\tilde{g}}_{k,i}\rVert^2+\tilde{e}_{k,i}\to 0\quad \text{as}\quad k\to\infty, \,\, (k,i)\in\mathcal{A}'. $$
	By contradiction, assume that there is $\bar{v}>0$ such that
	$$v_{k,i}\geq \bar{v}, \quad \text{for all} \,\, (k,i)\in\mathcal{A}. $$
	The convergence of the sequence $\{\bm x_k\}_{k\in\mathbb{N}_0}$ along with the boundedness of the sequence $\{\bm d_{k,i} \}_{(k,i)\in\mathcal{A}}$ yield the boundedness of the sequence $\{\bm x_k+ \bm d_{k,i} \}_{(k,i)\in\mathcal{A}}$. Now the locally Lipschitzness of $f$ implies the existence of $M>0$ such that
	\begin{align}
	\lvert f(\bm x_k+ \bm d_{k,i})-f(\bm x_k)\rvert&\leq M \lVert \bm d_{k,i}\rVert  \nonumber\\&
	\leq M \varepsilon^\alpha_k\, \lVert \bm{\tilde{g}}_{k,i}\rVert \label{Lip-inequality} \quad (\text{use} \, \eqref{Update-p3b}).
	\end{align}
	By Corollary \ref{Corollary2}, there is $C>0$ such that $\lVert \bm{\tilde{g}}_{k,i}\rVert\leq C$ for all $(k,i)\in\mathcal{A}$. Therefore
	$$\lvert f(\bm x_k+ \bm d_{k,i})-f(\bm x_k)\rvert\leq M \varepsilon^\alpha_k\,C. $$
	Since $\varepsilon^\alpha_k\downarrow0$ as $k\to\infty$, for $\bar{k}\in\mathbb{N}_0$ sufficiently large, we can write
	$$\lvert f(\bm x_k+ \bm d_{k,i})-f(\bm x_k)\rvert\leq \bar{v} \leq v_{k,i}= \frac{1}{2}\lVert \bm{\tilde{g}}_{k,i}\rVert^2+\tilde{e}_{k,i}, \quad \text{for all}\,\, (k,i)\in\mathcal{A}, \,\, k\geq \bar{k}.  $$
	The above inequality means that for every $k\geq \bar{k}$, the second part of the criterion \eqref{Criterion-i} holds, and therefore, Algorithm \ref{B-GS-Alg} does not reduce the sampling radius $\varepsilon^\alpha_k$ for such $k$. In other words, $\varepsilon^\alpha_k=\varepsilon^\alpha_{\bar k}$ for all $k\geq \bar k$, which contradicts $\varepsilon^\alpha_k\downarrow0$ as $k\to\infty$. Thus,  there is $\mathcal{A}'\subset\mathcal{A}$ such that
	$$v_{k,i}=\frac{1}{2}\lVert \bm{\tilde{g}}_{k,i}\rVert^2+\tilde{e}_{k,i}\to 0\quad \text{as}\quad k\to\infty, \,\, (k,i)\in\mathcal{A}', $$ 
	in other words
	$$\bm{\tilde{g}}_{k,i}\to 0 \quad \text{and} \quad \tilde{e}_{k,i}\to0 \quad \text{as} \quad k\to\infty, \,\, (k,i)\in\mathcal{A}'.  $$
	Now \eqref{tilde-e-g-in-process} and the upper semicontinuity of the set valued map $\partial_\cdot f(\cdot)$ yield $\bm 0\in\partial f(\bm{\tilde{x}})$.\\
	\emph{Case b.} There is $\bar{\varepsilon}>0$ such that $\varepsilon^\alpha_k\geq \bar{\varepsilon}$, for all $k$.
	
	We know by construction that there exists the infinite subset $\mathcal{K}$ of $\mathbb{N}_0$ such that for every $k\in\mathcal{K}$ the point $\bm x_{k+1}$ is updated by Line $16$ of Algorithm \ref{B-GS-Alg} (otherwise $\varepsilon^\alpha_k\downarrow 0$ as $k\to\infty$). For each $k\in\mathcal{K}$, there exists $i_k$ such that
	\begin{equation}\label{R1}
	f(\bm x_k+\bm d_{k,i_k})-f(\bm x_k)\leq \beta z_{k,i_k}, \quad \text{for all} \,\, k\in\mathcal{K}.
	\end{equation}
	Furthermore, for each $k\in\mathbb{N}_0\setminus\mathcal{K}$ we have $\bm x_{k+1}=\bm x_k$. Hence
	\begin{equation}\label{R2}
	f(\bm x_{k+1})-f(\bm x_k)=0, \quad \text{for all} \,\, k\in\mathbb{N}_0\setminus \mathcal{K}.
	\end{equation} 
	Using \eqref{R1} and \eqref{R2}, for an arbitrary $n\geq1$ one can write
	\begin{align*}
	f(\bm x_n)-f(\bm x_0)&=\left[f(\bm x_n)-f(\bm x_{n-1})]+\ldots+[f(\bm x_1)-f(\bm x_0)\right]\\&
	\leq \beta \sum_{\substack{k=1\\k\in\mathcal{K}}}^{n} z_{k,i_k}.
	\end{align*}
	Since $z_{k,i_k}\leq0$, letting $n$ approach infinity, we obtain
	$$-\infty<f(\bar{\bm x})-f(\bm x_0)\leq  \sum_{\substack{k=1\\k\in\mathcal{K}}}^{\infty} z_{k,i_k}\leq 0. $$
	This means that,
	$$z_{k,i_k}=-(\varepsilon^\alpha_k \lVert \bm{\tilde{g}}_{k,i_k}\rVert^2 +\tilde{e}_{k,i_k})\to 0 \quad \text{as} \quad k\to\infty , \, k\in\mathcal{K}. $$
	Eventually, since $\varepsilon^\alpha_k\geq \bar{\varepsilon}$ for every $k$, we conclude
	$$  \bm{\tilde{g}}_{k,i_k}\to \bm 0 \quad \text{and}\quad \tilde{e}_{k,i_k}\to 0 \quad \text{as} \quad k\to\infty , \, k\in\mathcal{K}.  $$
	Therefore, \eqref{tilde-e-g-in-process} and the upper semicontinuity of the set valued map $\partial_\cdot f(\cdot)$ imply $\bm 0\in\partial f(\bm{\tilde{x}})$. \qed
\end{proof}
Even if $f$ does not attain its minimum, we still have the following result.
\begin{theorem}
	If $X^*=\emptyset$, then 
	\[ f(\bm x_k)\downarrow\inf\{f(\bm x) \,\, : \,\, \bm x\in\mathbb{R}^n\}.  \]
\end{theorem}
\begin{proof}
	We know, by construction, the sequence $\{f(\bm x_k)\}_{k\in\mathbb{N}_0}$ is a decreasing sequence. Therefore, if the assertion is not true, there is $\bm x'\in\mathbb{R}^n$ such that $f(\bm x')\leq f(\bm x_k)$ for all $k\in\mathbb{N}_0$, which means that Assumption \ref{Assumption} holds. Hence, similar to Theorem \ref{Th1}, one can conclude that $\bm x_k\to\bm{\tilde{x}}\in X^*$. This contradicts $X^*=\emptyset$. \qed
\end{proof}

Next, we study the convergence of the method for \textit{Case II}. In other words, let us assume that there is $\bar{k}\in\mathbb{N}_0$ for which the inner loop of Algorithm \ref{B-GS-Alg} does not terminate. We start with the following lemma.

\begin{lemma}
	Assume that for a $\bar{k}\in\mathbb{N}_0$, the inner loop of Algorithm \ref{B-GS-Alg} does not terminate. Then, for every $i\in\mathbb{N}_0$, the following inequality holds
	\begin{equation}\label{INEQUALITY}
	\bigl\langle \nabla f(\bm s_{\bar k,m+i+1}), \bm{\tilde{g}}_{\bar k,i}\bigr\rangle \leq -\frac{1}{\varepsilon^\alpha_{\bar k}} e_{\bar k,m+i+1}+\beta \left( \lVert \bm{\tilde{g}}_{\bar k,i}\lVert^2+\frac{1}{\varepsilon^\alpha_{\bar k}} \tilde{e}_{\bar k,i} \right).
	\end{equation}
\end{lemma}
\begin{proof}
	For the sake of brevity in notations, we let $m_+$ denote $m+i+1$. First, we note that
	\begin{align*}
	\bm x_{\bar k}-\bm s_{\bar k,m_+}&=\bm x_{\bar k} - \texttt{FDP2}(\bm x_{\bar k}, \bm d_{\bar k,i}, \beta, z_{\bar k,i}, \sigma)\\&=\bm x_{\bar k}-(\hat{\bm x}_{\bar k,i}+\bm d_{\bar k,i})\\&
	=-(\hat{\bm x}_{\bar k,i}-\bm x_{\bar k})-\bm d_{\bar k,i}.
	\end{align*}
	Now, we have
	\begin{align*}
	e_{\bar k,m_+}&=f(\bm x_{\bar k})-\left[f(\bm s_{\bar k,m_+})+ \bigl\langle \nabla f(\bm s_{\bar k,m_+}), \bm x_{\bar k}-\bm s_{\bar k,m_+}\bigr\rangle\right] \nonumber\\&
	=f(\bm x_{\bar k})-f(\bm s_{\bar k,m_+})+ \bigl\langle \nabla f(\bm s_{\bar k,m_+}), \hat{\bm x}_{\bar k,i}-\bm x_{\bar k}\bigr\rangle+ \bigl\langle \nabla f(\bm s_{\bar k,m_+}), \bm d_{\bar k,i}\bigr\rangle \nonumber\\&
	\leq f(\bm x_{\bar k})-f(\bm s_{\bar k,m_+})+  \bigl\langle \nabla f(\bm s_{\bar k,m_+}), \bm d_{\bar k,i}\bigr\rangle +\lVert \nabla f(\bm s_{\bar k,m_+}) \rVert  \lVert  \hat{\bm x}_{\bar k,i}-\bm x_{\bar k}\rVert .
	\end{align*}
	Since we cam make $\lVert  \hat{\bm x}_{\bar k,i}-\bm x_{\bar k}\rVert$ arbitrarily small, the later inequality implies that
	\begin{align}
	e_{\bar k,m_+}& \leq f(\bm x_{\bar k})-f(\bm s_{\bar k,m_+})+  \bigl\langle \nabla f(\bm s_{\bar k,m_+}), \bm d_{\bar k,i}\bigr\rangle \nonumber\\&
	=f(\bm x_{\bar k})-f(\bm s_{\bar k,m_+})-\varepsilon^\alpha_{\bar k} \, \bigl\langle \nabla f(\bm s_{\bar k,m_+}),\bm{\tilde{g}}_{\bar k,i}\bigr\rangle \quad (\text{by} \,\, \eqref{Update-p3b}). \label{INEQ}
	\end{align}
	We recall that for every $i\in\mathbb{N}_0$, we have $f(\bm s_{\bar k,m_+})-f(\bm x_{\bar k})>\beta z_{\bar k,i} $. Therefore, we can continue \eqref{INEQ} with
	\begin{align*}
	e_{\bar k,m_+}& \leq -\beta z_{\bar k,i} -\varepsilon^\alpha_{\bar k} \, \bigl\langle \nabla f(\bm s_{\bar k,m_+}),\bm{\tilde{g}}_{\bar k,i}\bigr\rangle\\&
	= \beta \left( \varepsilon^\alpha_{\bar k} \, \lVert \bm{\tilde{g}}_{\bar k,i}\lVert^2+ \tilde{e}_{\bar k,i}  \right)-\varepsilon^\alpha_{\bar k} \, \bigl\langle \nabla f(\bm s_{\bar k,m_+}),\bm{\tilde{g}}_{\bar k,i}\bigr\rangle \quad (\text{by}\,\, \eqref{Update-p3a}),
	\end{align*}
	and after a simple manipulation, we obtain \eqref{INEQUALITY}.  \qed
	
	\begin{lemma}\label{Bounded-seq}
		Assume that for a $\bar{k}\in\mathbb{N}_0$ the inner loop of Algorithm \ref{B-GS-Alg} does not terminate. Then the sequences $\{w_{\bar k, i}\}_i$, $\{\bm{\tilde{g}}_{\bar k,i}\}_i$ and $\{\bm d_{\bar k,i}\}_i$ $(i\in\mathbb{N}_0)$ are bounded.	
	\end{lemma}
	\begin{proof}
		In the proof of Lemma \ref{Lemma2}, replace $k$ by $\bar{k}$ to see
		$$w_{\bar k,i}\leq \frac{1}{2}\lVert \nabla f(\bm x_{\bar k})\rVert, \quad \text{for all} \,\, i\in \mathbb{N}_0, $$
		which means that $\{w_{\bar k, i}\}_{i}$ is bounded. Now similar to Corollary \ref{Corollary2}, the boundedness of sequences $\{\bm{\tilde{g}}_{\bar k,i}\}_i$ and $\{\bm d_{\bar k,i}\}_i$ follows from the following relations 
		\begin{align*}
		&w_{\bar k,i}=\frac{1}{2} \lVert \bm{\tilde{g}}_{\bar k,i}\rVert^2+\tilde{e}_{\bar k,i}, \quad \text{for all} \,\, i\in\mathbb{N}_0, \\&
		\bm d_{\bar k,i}=-\varepsilon^{\alpha}_{\bar k}\, \bm{\tilde{g}}_{\bar k,i}, \quad \text{for all} \,\, i\in\mathbb{N}_0.
		\end{align*}\qed

	\end{proof}
	
\end{proof}
Now we are ready to state the main theorem for \emph{Case II}.
\begin{theorem}\label{Theorem3}
	Assume that for a $\bar{k}\in\mathbb{N}_0$ the inner loop of Algorithm \ref{B-GS-Alg} does not terminate. Then
	$$\bm 0\in\partial f(\bm x_{\bar k}).$$
\end{theorem}	
\begin{proof}
	
	First, we prove that
	$$w_{\bar k,i+1}\leq w_{\bar k,i}, \quad \text{for all}\,\, i\in\mathbb{N}_0. $$
	We know that $w_{\bar k,i+1}$ is the optimal value of subproblem \eqref{Dual-sub-in process}. Let $t\in[0,1]$ be arbitrary and define the vector $\boldsymbol{\lambda}=(\lambda_j; j\in J_{\bar k,i}, \tilde{\lambda})$ by 
	$$\lambda_j=0, \,\, j\in J_{\bar k,i}\setminus\{m+i+1\}, \,\,\,\, \lambda_{m+i+1}=t, \,\,\,\, \tilde{\lambda}=1-t. $$
	Then $\boldsymbol{\lambda}$ is feasible for subproblem \eqref{Dual-sub-in process}, and hence
	\begin{align*}
	w_{\bar k,i+1}&\leq \frac{1}{2}\lVert t \nabla f(\bm s_{\bar k,m+i+1})+ (1-t) \bm{\tilde{g}}_{\bar k,i}\rVert^2+ \frac{1}{\varepsilon^\alpha_{\bar k}} t\, e_{\bar k,m+i+1}+ \frac{1}{\varepsilon^\alpha_{\bar k}} (1-t)\tilde{e}_{\bar k,i} \nonumber \\&
	=:Q(t).
	\end{align*}
	Therefore
	\begin{equation}\label{K1}
	w_{\bar k,i+1}\leq Q(t), \quad \text{for all}\,\,t\in[0,1].
	\end{equation}
	Simple calculations yield
	\begin{align}
	Q(t)&=\frac{1}{2} t^2 \lVert \nabla f(\bm s_{\bar k,m+i+1})- \bm{\tilde{g}}_{\bar k,i}\rVert^2 + t \left(\bigl\langle \nabla f(\bm s_{\bar k,m+i+1}), \bm{\tilde{g}}_{\bar k,i}\bigr\rangle - \lVert  \bm{\tilde{g}}_{\bar k,i}\rVert^2\right) \nonumber \\&
	\qquad +\frac{1}{2}  \lVert  \bm{\tilde{g}}_{\bar k,i}\rVert^2 + \frac{1}{\varepsilon^\alpha_{\bar k} }\tilde{e}_{\bar k,i} + \frac{1}{\varepsilon^\alpha_{\bar k}} t (e_{\bar k,m+i+1}-\tilde{e}_{\bar k,i}) \nonumber\\&
	= \frac{1}{2} t^2 \lVert \nabla f(\bm s_{\bar k,m+i+1})- \bm{\tilde{g}}_{\bar k,i}\rVert^2 + t \left(\bigl\langle \nabla f(\bm s_{\bar k,m+i+1}), \bm{\tilde{g}}_{\bar k,i}\bigr\rangle - \lVert  \bm{\tilde{g}}_{\bar k,i}\rVert^2\right) \nonumber\\&
	\qquad +w_{k,i}+  \frac{1}{\varepsilon^\alpha_{\bar k}} t (e_{\bar k,m+i+1}-\tilde{e}_{\bar k,i}) \quad (\text{by}\,\, \eqref{optimal-value-of-Dualsub}).  \label{EQUALITY}
	\end{align}
	Using inequality \eqref{INEQUALITY}, we can continue \eqref{EQUALITY} with
	\begin{align}
	Q(t)&\leq \frac{1}{2} t^2 \lVert \nabla f(\bm s_{\bar k,m+i+1})- \bm{\tilde{g}}_{\bar k,i}\rVert^2 \nonumber \\&
	\qquad + t \left( -\frac{1}{\varepsilon^\alpha_{\bar k}} e_{\bar k,m+i+1}+ \beta \lVert \bm{\tilde{g}}_{\bar k,i}\rVert^2+ \frac{\beta}{\varepsilon^\alpha_{\bar k}} \tilde{e}_{\bar k,i}- \lVert \bm{\tilde{g}}_{\bar k,i}\rVert^2\right) \nonumber \\ &
	\qquad \quad+w_{\bar k,i}+  \frac{t}{\varepsilon^\alpha_{\bar k}}  (e_{\bar k,m+i+1}-\tilde{e}_{\bar k,i}) \nonumber \\&
	= \frac{1}{2} t^2 \lVert \nabla f(\bm s_{\bar k,m+i+1})- \bm{\tilde{g}}_{\bar k,i}\rVert^2-t(1-\beta)\left( \lVert \bm{\tilde{g}}_{\bar k,i}\rVert^2 +\frac{1}{\varepsilon^\alpha_{\bar k}} \tilde{e}_{\bar k,i}   \right)\nonumber  \\&
	\qquad -\frac{t}{\varepsilon^\alpha_{\bar k}}(e_{\bar k,m+i+1}-\tilde{e}_{\bar k,i})+w_{\bar k,i}+\frac{t}{\varepsilon^\alpha_{\bar k}}  (e_{\bar k,m+i+1}-\tilde{e}_{\bar k,i}) \nonumber \\&
	\leq \frac{1}{2} t^2 \lVert \nabla f(\bm s_{\bar k,m+i+1})- \bm{\tilde{g}}_{\bar k,i}\rVert^2 - t(1-\beta)w_{\bar k,i}+w_{\bar k,i} \quad (\text{by}\,\, \eqref{optimal-value-of-Dualsub}).  \nonumber
	\end{align}
	Now, if we set
	\begin{equation*}
	C_{\bar k,i}:=\max \{\lVert \bm{\tilde{g}}_{\bar k,i}\rVert, \lVert \nabla f(\bm s_{\bar k,m+i+1})\rVert, \frac{1}{\varepsilon^\alpha_{\bar k}}  \tilde{e}_{\bar k,i}, 1  \},
	\end{equation*}
	then, we observe
	\begin{equation}\label{K2}
	Q(t)\leq 2t^2C_{\bar{k},i}^2 -t(1-\beta)w_{\bar k,i}+w_{\bar k,i}=: \phi(t), \quad \text{for all} \,\, t\in[0,1].
	\end{equation}
	It is easy to check that
	$$\bar{t}:=\frac{(1-\beta)w_{\bar k,i}} {4C_{\bar k,i}^2}\leq \frac{(1-\beta)}{4C_{\bar k,i}^2}\left( \frac{C_{\bar k,i}^2}{2}+C_{\bar k,i}\right)\leq 1,  $$
	minimizes $\phi(t)$ and
	\begin{equation}\label{K3}
	\phi(\bar t)=w_{\bar k,i}-\frac{(1-\beta)^2w_{\bar k,i}^2}{8 C_{\bar k,i}^2}. 
	\end{equation}
	Consequently, in view of \eqref{K1}, \eqref{K2} and \eqref{K3}, we have
	\begin{equation}\label{main}
	w_{\bar k,i+1}\leq Q(\bar t) \leq \phi(\bar t)=w_{\bar k,i}-\frac{(1-\beta)^2w_{\bar k,i}^2}{8 C_{\bar k,i}^2}.
	\end{equation}  
	This means that
	\begin{equation}\label{DR}
	0\leq w_{\bar k,i+1}\leq w_{\bar k,i},\quad \text{for all} \,\, i\in\mathbb{N}_0. 
	\end{equation}
	Thus, the sequence $\{w_{\bar k, i}\}_{i\in\mathbb{N}_0}$ is decreasing and bounded from below, and therefore, it converges. Next, we show that $w_{\bar k,i}\downarrow0$ as $i\to\infty$.
	A glance at \eqref{DR}  shows that
	\begin{equation*}
	w_{\bar k,i}= \frac{1}{2} \lVert \bm{\tilde{g}}_{\bar k,i}\rVert^2+ \frac{1}{\varepsilon^\alpha_{\bar k}} \tilde{e}_{\bar k,i}\leq w_{\bar k,0},\quad \text{for all} \,\, i\geq 1,
	\end{equation*}
	which implies the existence of $C_1>0$, such that
	\begin{equation}\label{F1}
	\max \{\lVert \bm{\tilde{g}}_{\bar k,i}\rVert, \frac{1}{\varepsilon^\alpha_{\bar k}}  \tilde{e}_{\bar k,i}, 1  \}\leq C_1, \quad \text{for all} \,\, i\in\mathbb{N}_0.
	\end{equation}
	Furthermore
	\begin{equation*}
	\bm s_{\bar k,m+i+1}=\texttt{FDP2}(\bm x_{\bar k}, \bm d_{\bar k,i}, \beta, z_{\bar k,i}, \sigma)=\bm{\hat{x}}_{\bar k,i}+\bm d_{\bar k,i}.
	\end{equation*}
	We know that $\bm{\hat{x}}_{\bar k,i}\in B(\bm x_{\bar k},\sigma)$, and hence the boundedness of sequence $\{\bm d_{\bar k,i}\}_{i\in\mathbb{N}_0}$ yields the boundedness of the sequence $\{ \bm s_{\bar k,m+i+1} \}_{i\in\mathbb{N}_0}$. Now the local boundedness of the map $\nabla f:\mathbb{R}^n\to\mathbb{R}^n$ implies the existence of $C_2>0$ such that
	\begin{equation}\label{F2}
	\lVert\nabla f(\bm s_{\bar k,m+i+1})\rVert\leq C_2, \quad \text{for all} \,\, i\in\mathbb{N}_0. 
	\end{equation}
	Therefore, if we set $C_{\bar{k}}:=\max\{C_1, C_2\}$, in view of \eqref{F1} and \eqref{F2}, we have
	$$\max \{\lVert \bm{\tilde{g}}_{\bar k,i}\rVert, \lVert \nabla f(\bm s_{\bar k,m+i+1})\rVert, \frac{1}{\varepsilon^\alpha_{\bar k}} \tilde{e}_{\bar k,i}, 1  \}\leq  C_{\bar k}, \quad \text{for all} \,\, i\in\mathbb{N}_0, $$  
	and hence \eqref{main} is rewritten by 
	\begin{equation*}
	w_{\bar k,i+1}\leq w_{\bar k,i}-\frac{(1-\beta)^2 w_{\bar k,i}^2}{8 C_{\bar k}^2}, \quad \text{for all} \,\, i\in\mathbb{N}_0.
	\end{equation*} 
	Now assume $w_{\bar k,i}\downarrow A$ as $i\to \infty$. Let $i$ approach infinity in the above inequality to see
	$$A\leq A-\frac{(1-\beta)^2A^2}{8 C_{\bar k}^2}. $$
	Since $\beta\in(0,1)$, the later inequality yields $A=0$.
	Eventually, we observe that
	$$w_{\bar k,i}= \frac{1}{2} \lVert \bm{\tilde{g}}_{\bar k, i}\rVert^2 + \frac{1}{\varepsilon^\alpha_{\bar k}} \tilde{e}_{\bar k,i}\to 0, \quad \text{as} \quad i\to\infty, $$
	which means that $\bm{\tilde{g}}_{\bar k, i}\to \bm 0$ and $\tilde{e}_{\bar k,i}\to0$ as $i\to\infty$. Now from \eqref{tilde-e-g-in-process} and the upper semicontinuity of the map $\partial_\cdot f(\cdot)$,  we conclude $\bm 0\in\partial f(\bm x_{\bar k})$. \qed
\end{proof}

The assumptions on which the convergence of the method was established were much weaker than the GS like methods. In particular, the objective $f$ need not be continuously differentiable on an open set with full measure in $\mathbb{R}^n$. Moreover,
we had no assumption on the sample size $m\in \mathbb N$. Thus,
only $\mathcal O(1)$ gradient evaluations to establish the initial polyhedral model, and a single gradient evaluation at each iteration of the improvement process ensure the convergence of the algorithm.

\begin{remark}\label{Remark1}
	Through some modifications,	for the class of piecewise linear functions, one can establish finite convergence for Algorithm \ref{B-GS-Alg}. Indeed, we need to follow the \emph{subgradient selection strategy} described in \cite{kiwielbook} to modify the updating rules of the index set $J_{k,i}$, such that it preserves the index of positive Lagrangian multipliers. Furthermore, we need to assume that some \emph{Haar} conditions are satisfied, and the objective function $f$ is bounded from below. For the sake of brevity, we omit to provide a detailed description, and we refer the reader to the second chapter of \cite{kiwielbook} for a comprehensive discussion. However, in our numerical experiments, we investigate the behavior of the B-GS algorithm on some piecewise linear function.
\end{remark}

\section{Numerical results}\label{Numerical results}
In this section, we asses the efficiency of the proposed method over a variety of convex test problems. To this end, we divide this section into four experiments. In the first experiment, some important features of the proposed method are illustrated. Next, we observe the finite convergence of our method for a piecewise linear objective. In the third experiment, we examine the sensitivity of the proposed method to the accuracy of supplied gradients, and in the last experiment, we conduct a comparison between some GS type methods.

In the following experiments, we use a set of nonsmooth convex test problems given in Table \ref{Table1}. Note that Problems 1-8 are scalable, in the sense that they can be formulated with any number of variables, while Problems 9-13 are some small scale problems that are used to provide some illustrations of the proposed method.  All experiments have been implemented in  \textsc{Matlab} software on PC Intel Core i5 2700k CPU 3.5 GHz and 8 GB of RAM.

\begin{table}[ht]
	\centering
	\caption{List of test problems}\label{Table1}
	%\vspace{.001pt}

	\begin{tabular}{ c c c c c  } 
		
		\toprule[1pt]
		Problem&Name&$n$ & $f^*$ & Ref.  \\
		\hline \\
		1& Tilted Norm Function&any & $0$& \cite{Skaaja} \\
		2& Generalization of MXHILB  &any& $0$ & \cite{Haarala}  \\ 
		3& Chained LQ  &any& $-(n-1)\sqrt{2}$ & \cite{Haarala}   \\ 
		4& Chained CB3 I &any & $2(n-1)$&  \cite{Haarala}  \\ 
		5& Chained CB3 II &any& $2(n-1)$ & \cite{Haarala} \\
		6& Generalization of MAXQ & any  & $0$ & \cite{Haarala}  \\
		7& Generalization of MAXL& any & $0$& \cite{techreport2}\\
		8&  A Convex Partly Smooth Function& any   & $0$ & \cite{Skaaja}  \\
		9& QL & $2$ & $7.2$ & \cite{Makela_book}  \\ 
		10& Mifflin1 & $2$ & $-1$& \cite{Makela_book} \\
		11& MAXQ & $20$ & $0$  & \cite{Makela_book} \\
		12& Goffin& $50$& $0$  & \cite{Makela_book} \\
		13& Rosen& $4$  & $-44$& \cite{Makela_book} \\

		\bottomrule[1pt]
		
	\end{tabular}
	
\end{table}

By knowing the role of each parameter in the B-GS algorithm, the user can feel free to set them; however, the following choices are some safe values for the input parameters in order to observe acceptable behaviors in practice.

\begin{table}[H]
	\centering
	\begin{tabular}{ c c c c c c c c }

		Parameter &$\varepsilon_0$ & $\mu$ & $\alpha$ & $\gamma$& $m$& $\beta$& $\theta$  \\
		\hline \\
		Value & $1$& $0.5$& $0.5$& $0.9$& $\big[ \lceil n/10\rceil, 2n \big]$& $10^{-6}$& $0.9$

		%\bottomrule
	\end{tabular}
\end{table}
As we already mentioned, the quality of initial polyhedral models is dependent on the sample size $m\in\mathbb{N}$. Typically, a large sample size increases the quality of the model, but at the cost of increased computation time. In this regard, if the size of a problem is small and gradient evaluations are not expensive, the user can work with a large sample size (e.g., $m=2n$). Otherwise, the user has to choose a smaller value for this input parameter. Accordingly, in Experiments 1-3 in which some small scale problems are considered, we set $m=2n$. On the other hand, in Experiment 4, where we consider some medium and large scale problems, we set $m=n/10$. 

We also note that, as reported in \cite{Burke2005}, in the implementation of GS type methods, the differentiability check need not be taken into account. Accordingly, in our implementations of Algorithm \ref{B-GS-Alg}, the differentiability checks are omitted. More precisely, the criterion in Line 4 is never checked, and instead of using \texttt{FDP1} and \texttt{FDP2} functions in Lines 16 and 19 , $\bm x_{k+1}$ and $\bm s_{k,m+i+1}$ are simply defined by $\bm x_k+\bm d_k$ and $\bm x_k+\bm d_{k,i}$, respectively.

\subsection*{Experiment 1.}
In this experiment, we present some illustrations of the proposed method, which reveals different aspects of the B-GS algorithm.  Figure~\ref{fig:figa} shows the iterations path of the B-GS method for QL and Mifflin1 problems (Problems 9 and 10 in Table \ref{Table1}). As we can see from this figure, when the method reaches the nonsmooth region, it starts to track a nonsmooth curve towards the minimum point, and the iterates do not fluctuate around the nonsmooth curves. 

\begin{figure}
	\centering
	\includegraphics[width=\linewidth]{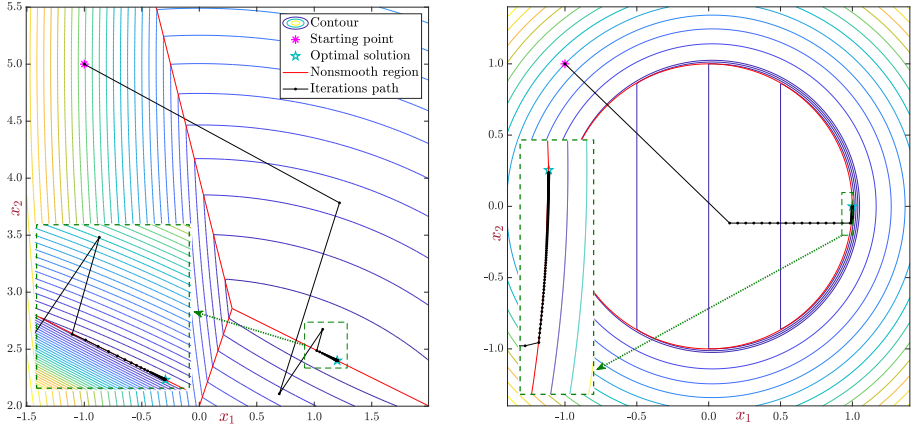}
	\caption{Contour plots and iterations path of the B-GS method, for QL (left) and Mifflin1 (right) problems.}
	%Generated by ForFigA.m and boundle_GSv3.m.
	\label{fig:figa}
\end{figure}

In Figure \ref{fig:figb}, the left vertical axes represent the function trials generated by the B-GS method using different initial sampling radii $\varepsilon_0=4$ (upper plot) and $\varepsilon_0=2^{-3}$ (lower plot), on MAXQ problem. At the same time, the right axes show  the current value of the sampling radius.
Moreover, to show the role of inner and outer iterations of Algorithm \ref{B-GS-Alg}, they are indicated by different symbols. 
In this way, Figure \ref{fig:figb} illustrates how the improvement process plays its role to
create a substantial reduction in the objective function $f$.
As observed from the plots of this figure, if our initial polyhedral model is not an adequate approximation to $f$, the improvement process starts to enhance the quality of the model either by reducing the sampling radius or by enriching the polyhedral model. Moreover, since, at the beginning of each iteration, the GS technique generates a rich set of auxiliary points, we see that after a few numbers of inner iterations, a reduction occurs in the objective function $f$.

\begin{figure}
	\centering
	\includegraphics[width=\linewidth]{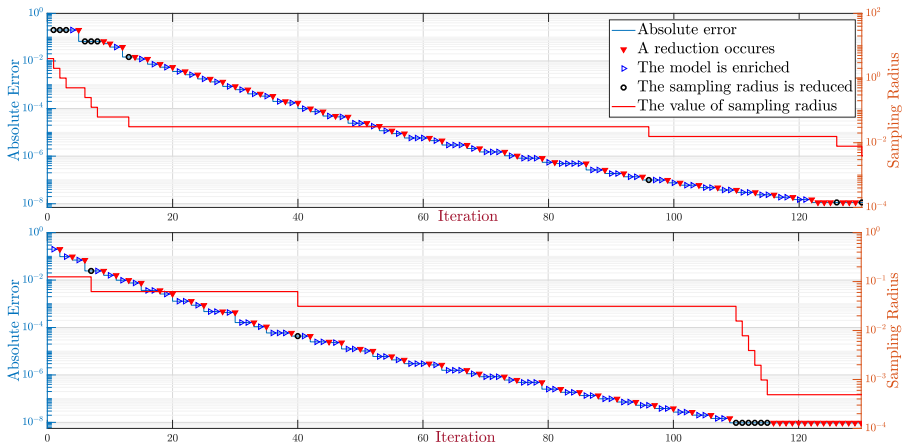}
	\caption{Top: (left axis) The role of inner and outer iterations in the B-GS algorithm using $\varepsilon_0=4$ against the absolute error $ f(\bm x_{k})-f^*$, for MAXQ problem. (right axis)  The current value of the sampling radius. 
		Bottom: The same for $\varepsilon_0=2^{-3}$.	}
	\label{fig:figb}
\end{figure}

It can also be seen from the upper plot of Figure \ref{fig:figb} that, If the initial sampling radius $\varepsilon_0$  is chosen to be large (here $\varepsilon_0=4$), 
the method smartly reduces the value of this input parameter, until it finds the efficient one. Next, the method starts its normal process to reduce the value of the objective function. As a result, the user need not be worried about choosing a large initial sampling radius because the B-GS method smartly finds the suitable one. 

A glance at last iterations of Figure \ref{fig:figb} reveals that when the method reaches its final accuracy (here $10^{-8}$), reducing the sampling radius does not make a significant reduction in the objective function $f$. Indeed, we only see very short steps towards the minimum point. In this regard, as in GS type methods, it is reasonable to terminate the algorithm when the sampling radius becomes smaller than a prespecified tolerance.

\subsection*{Experiment 2.}
As mentioned in Remark \ref{Remark1}, under some subtle modifications, the B-GS method is finitely convergent for the class of piecewise linear functions. From numerical point of view, the B-GS method minimizes such problems with high accuracy (close to the machine accuracy),  and generally, a big jump to the minimum point occurs.
To observe such  behavior in practice, we consider the piecewise linear problem Goffin (Problem 12 in Table \ref{Table1}), which is defined by
$$G(\bm x)=\max_{1\leq i\leq 50} x_i-\sum_{i=1}^{50}x_i. $$
Figure \ref{fig:figc} demonstrates how the B-GS method treats the Goffin problem. It can be seen from this figure that, after a number of iterations, a big jump to the minimum point occurs, which confirms the finite convergent property of the proposed method for the class of piecewise linear functions.

\begin{figure}[h]
	\centering
	\includegraphics[width=\linewidth]{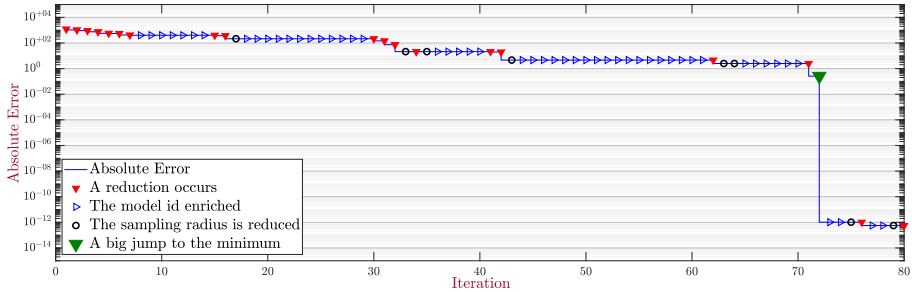}
	\caption{The role of inner and outer iterations of the B-GS method  against the absolute error $ f(\bm x_{k})-f^*$ when applied to the Goffin test problem, which shows a big jump to the minimum point. 	}
	\label{fig:figc}
\end{figure}

\subsection*{Experiment 3.}
In many of nonsmooth optimization methods (e.g., bundle methods) the accuracy of supplied (sub)gradients is an extremely delicate issue. In this experiment, we intend to examine the sensitivity of the B-GS method to the accuracy of gradients. To this end, let us apply the B-GS method to the Rosen problem (Problem 13 in Table \ref{Table1}) by considering two cases. In the first case, using the analytical form of the gradient map, we supply the gradients exactly, while in the second case, we use the \emph{forward difference formula} with $h=10^{-9}$. The upper plots of Figure \ref{fig:fig0} indicate the function trials of the B-GS method using exact (left) and approximate (right) gradients, for five randomly generated starting points. It is observed that even when a simple difference formula supplies gradients, the B-GS method still provides an accurate approximation of the minimum point. This can be attributed to the fact that the GS technique provides an adequate knowledge of the $\varepsilon_k$-subdifferential at the beginning of each iteration. Of course, as seen from the left plot, using exact gradients results in more accurate solutions. 

To make a comparison,  the same experiment is performed for the Bundle-Trust (BT \cite{BT-method}) method. The obtained results have been illustrated in the lower plots of Figure \ref{fig:fig0}.  As the left plot shows, the method works well as long as the gradients are supplied exactly. However, it can be seen from the right plot that, the BT method fails to make significant progress towards the minimum point when a simple difference formula supplies gradients. This is due to the fact that the knowledge of the BT method from the nearby nonsmooth curve(s) is not rich enough to handle a margin of gradient errors.\\

\begin{figure}
	\centering
	\includegraphics[width=\linewidth]{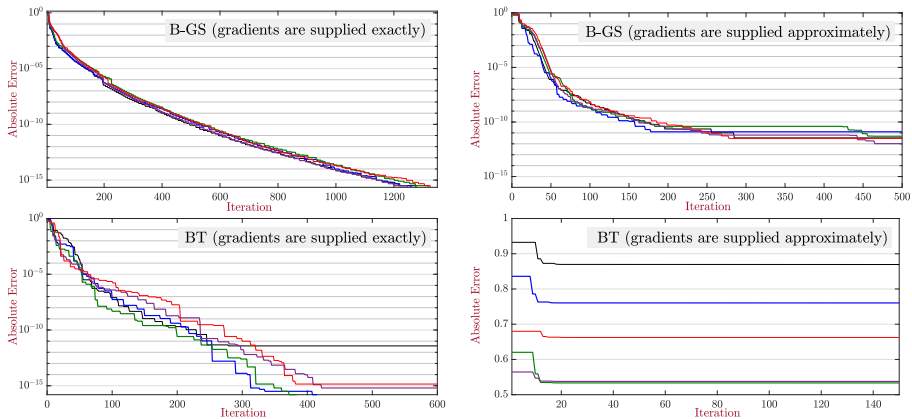}
	\caption{Upper plots: function trials generated by the B-GS method using exact (left) and approximate (right) gradients, for five randomly generated starting points. Lower plots: the same for the Bundle Trust (BT) method. }
	%Generated by ForFigA.m and boundle_GSv3.m.
	\label{fig:fig0}
\end{figure}

\subsection*{Experiment 4.}
In this experiment, we examine the performance of the B-GS, AGS and GS methods on Problems 1-8. We recall that all of these problems can be formulated with any number of variables. However, in this experiment, we chose $n=50, 100, 200$ and $500$. At first, we briefly review the GS and AGS methods. \\

\textbf{Gradient Sampling Method} \cite{Burke2005}. The main strategy of the GS method is to use the set of sampled points to approximate the $\varepsilon$-steepest descent direction. Then, a standard Armijo line search is applied to make a substantial reduction in the objective function. This method is robust and can be applied to a broad range of nonsmooth problems, but to guarantee the convergence of the method, at each iteration, we need a set of sampled points of size $n+1$(or more). In particular, as reported in \cite{Burke2005} and used in our experiments, to observe nice practical behavior, we need to set the size of the sample to $2n$, which makes the method computationally extensive. In this experiment, we used the \textsc{Matlab} code of the GS algorithm which is freely available (see \cite{Burke2005}). \\

\textbf{Adaptive Gradient Sampling Method} \cite{Curtis2013}. The main idea of the Adaptive Gradient Sampling (AGS) method is to use the set of sampled points to approximate the Hessian of $f$ at the current point in order to construct a quadratic model for the objective function $f$. The method uses this quadratic model to generate a search direction, and then a standard Armijo line search checks the efficiency of this direction.  One important feature of the AGS method is that the size of the sample is independent of the dimension of a problem.  Several versions of the AGS algorithm has been proposed in \cite{Curtis2013}. In this experiment, we consider the AGS-LBFGS variant, which has a guaranteed convergence analysis. Following \cite{Curtis2013}, we set  the size of the sample to $ n/10$. In this experiment, the \textsc{Matlab} code of the AGS-LBFGS algorithm has been written by the authors. \\

\setlength{\tabcolsep}{1pt}			
\begin{table}[h]
	\centering
	\caption{Numerical results for Experiment 4.}\label{Table2}
	
	\begin{tabular}{cllccclccclccclccc}
		\toprule[1pt]
		\rule{0pt}{3ex}&  &  & \multicolumn{3}{c}{Iters} &  & \multicolumn{3}{c}{$g_{\rm{eval}}$} &  & \multicolumn{3}{c}{Time(s)} &  & \multicolumn{3}{c}{$E_{\rm{final}}$} \\ 
		\cline{4-6}\cline{8-10}\cline{12-14}\cline{16-18}
		\rule{0pt}{3ex}Problem & \,$n$ &  & B-GS & AGS & GS &  & B-GS & AGS & GS &  & B-GS & AGS & GS &  & B-GS & AGS & GS \\ 
		\cline{1-2}\cline{4-6}\cline{8-18}
		\rule{0pt}{3ex}1 & 50 &  & 6 & 31 & 86 &  & 39 & 183 & 8600 &  & 0.6 & 0.7 & 2.3 &  & 2e-6 & 2e-4 & 2e-4 \\ 
		2 & 50 &  & 262 & 225 & 72 &  & 3189 & 1485 & 7250 &  & 11.4 & 3.0 & 2.4 &  & 4e-4 & 4e-4 & 3e-4 \\ 
		3 & 50 &  & 22 & 57 & 80 &  & 240 & 340 & 8050 &  & 0.8 & 1.1 & 3.4 &  & 4e-4 & 4e-4 & 4e-4 \\ 
		4 & 50 &  & 31 & 46 & 60 &  & 221 & 273 & 6050 &  & 1.2 & 1.5 & 6.1 &  & 4e-4 & 4e-4 & 2e-4 \\ 
		5 & 50 &  & 32 & 35 & 36 &  & 220 & 210 & 3600 &  & 0.8 & 0.9 & 3.6 &  & 4e-4 & 3e-4 & 3e-4 \\ 
		6 & 50 &  & 102 & 209 & 19 &  & 612 & 1252 & 1950 &  & 1.1 & 2.2 & 0.6 &  & 4e-4 & 4e-4 & 2e-4 \\ 
		7 & 50 &  & 63 & 118 & 51 &  & 419 & 706 & 5050 &  & 1.0 & 1.1 & 1.2 &  & 3e-4 & 4e-4 & 4e-4 \\ 
		8 & 50 &  & 12 & 14 & 72 &  & 73 & 82 & 7250 &  & 0.4 & 0.5 & 2.1 &  & 1e-6 & 4e-4 & 4e-4 \\ 
		\hline
		\rule{0pt}{3ex}1 & 100 &  & 8 & 41 & 102 &  & 94 & 449 & 20500 &  & 0.6 & 1.2 & 6.8 &  & 2e-5 & 4e-4 & 2e-4 \\ 
		2 & 100 &  & 341 & 532 & 101 &  & 7050 & 5769 & 20200 &  & 16.1 & 9.2 & 11.5 &  & 4e-4 & 4e-4 & 4e-4 \\ 
		3 & 100 &  & 24 & 36 & 89 &  & 268 & 395 & 17900 &  & 1.6 & 2.0 & 12.0 &  & 4e-4 & 4e-4 & 4e-4 \\ 
		4 & 100 &  & 43 & 75 & 262 &  & 740 & 823 & 52500 &  & 3.9 & 11.6 & 129.5 &  & 4e-4 & 3e-4 & 4e-4 \\ 
		5 & 100 &  & 27 & 24 & 19 &  & 323 & 263 & 3800 &  & 1.3 & 1.3 & 9.7 &  & 4e-4 & 3e-4 & 3e-4 \\ 
		6 & 100 &  & 207 & 171 & 22 &  & 2277 & 1878 & 4500 &  & 2.0 & 2.1 & 1.3 &  & 4e-4 & 4e-4 & 4e-4 \\ 
		7 & 100 &  & 108 & 152 & 47 &  & 1265 & 1670 & 9400 &  & 1.1 & 1.5 & 2.5 &  & 3e-4 & 4e-4 & 4e-4 \\ 
		8 & 100 &  & 12 & 17 & 112 &  & 133 & 183 & 22500 &  & 0.4 & 0.6 & 9.0 &  & 1e-6 & 4e-4 & 4e-4 \\ 
		\hline
		\rule{0pt}{3ex}1 & 200 &  & 12 & 39 & - &  & 260 & 816 & - &  & 0.7 & 2.1 & - &  & 3e-5 & 2e-4 & - \\ 
		2 & 200 &  & 794 & 1027 & - &  & 26760 & 21421 & - &  & 114.4 & 66.5 & - &  & 4e-4 & 4e-4 & - \\ 
		3 & 200 &  & 27 & 49 & - &  & 694 & 1028 & - &  & 3.5 & 8.6 & - &  & 4e-4 & 4e-4 & - \\ 
		4 & 200 &  & 38 & 40 & - &  & 816 & 829 & - &  & 14.3 & 19.2 & - &  & 4e-4 & 4e-4 & - \\ 
		5 & 200 &  & 33 & 35 & - &  & 711 & 734 & - &  & 7.1 & 8.7 & - &  & 4e-4 & 4e-4 & - \\ 
		6 & 200 &  & 415 & 335 & - &  & 8736 & 7049 & - &  & 8.6 & 7.4 & - &  & 4e-4 & 4e-4 & - \\ 
		7 & 200 &  & 193 & 205 & - &  & 4163 & 4303 & - &  & 4.1 & 4.8 & - &  & 2e-4 & 4e-4 & - \\ 
		8 & 200 &  & 12 & 13 & - &  & 253 & 271 & - &  & 0.7 & 0.9 & - &  & 3e-6 & 2e-4 & - \\ 
		\hline
		\rule{0pt}{3ex}1 & 500 &  & 18 & 23 & - &  & 927 & 1172 & - &  & 4.9 & 23.7 & - &  & 1e-4 & 4e-3 & - \\ 
		2 & 500 &  & 865 & 657 & - &  & 46695 & 33436 & - &  & 834.5 & 1051.3 & - &  & 4e-3 & 4e-3 & - \\ 
		3 & 500 &  & 10 & 11 & - &  & 518 & 561 & - &  & 11.8 & 14.7 & - &  & 3e-3 & 4e-3 & - \\ 
		4 & 500 &  & 20 & 16 & - &  & 1057 & 816 & - &  & 67.3 & 55.4 & - &  & 4e-3 & 4e-3 & - \\ 
		5 & 500 &  & 15 & 29 & - &  & 776 & 1477 & - &  & 45.8 & 135.2 & - &  & 4e-3 & 3e-3 & - \\ 
		6 & 500 &  & 597 & 362 & - &  & 30449 & 18458 & - &  & 119.1 & 89.8 & - &  & 4e-3 & 4e-3 & - \\ 
		7 & 500 &  & 196 & 344 & - &  & 10267 & 17542 & - &  & 42.7 & 69.4 & - &  & 4e-3 & 4e-3 & - \\ 
		8 & 500 &  & 6 & 7 & - &  & 313 & 356 & - &  & 3.4 & 6.6 & - &  & 1e-3 & 1e-3 & - \\ 
		\bottomrule[1pt]
	\end{tabular}
\end{table}

Now, we turn to the details of our simulations. 
We stop an algorithm whenever one of the following conditions is met
\begin{itemize}
	\item[1.] If the relative error $$E_k:=\frac{f(\bm x_k)-f^*}{\lvert f^*\rvert +1},$$
	becomes smaller than some  positive tolerance $\epsilon>0$. In this experiment, we chose $\epsilon=5\times10^{-4}$ for $n\leq200$, and $\epsilon=5\times10^{-3}$ for $n=500$.
	\item[2.] If the number of iterations exceeds 1000. Note that, for the B-GS method, the number of inner iterations are not taken into account.
	\item[3.] If the value of the sampling radius becomes smaller than $10^{-12}$.
\end{itemize}
Due to the stochastic nature of the algorithms, we run each problem five times using a starting point randomly generated from a ball centered at $\bm x_0$ (suggested in the literature) with radius $(\lVert \bm x_0\rVert+1)/n$, and the average of the results are reported.

Table \ref{Table2} compares the obtained results for the considered set of test problems. In this table, the final relative error $E_k$ and the number of gradient evaluations are denoted by $E_{\rm{final}}$ and $g_{\rm{eval}}$, respectively. These results clearly demonstrate that the B-GS method outperforms the two other methods in the sense of the number of gradient evaluations and CPU Time. We also observe that, in the GS method,  as the number of variables increases the number of gradient evaluations increases significantly. This feature makes this method inefficient for solving large scale problems. In this regard, we did not apply this method to the problems with $n>100$.

\section{Conclusion}\label{Conclusion}
In this paper, we combined the effective and important features of the bundle and GS approaches, to develop  a variant of the bundle methods for minimizing  nonsmooth convex functions. 
%-------------
Moreover,  we studied the convergence of the proposed method for the class of convex functions under the assumptions that are much weaker than those in the GS type methods.
%-----
The presented method applied to a variety of benchmark nonsmooth convex problems and a comprehensive report of the numerical results was presented.
%-----
Based on these results, one can find that the proposed method requires much fewer gradient evaluations in comparison with GS type methods. 
Furthermore, we observed that the method is able to locate the minimizer even using approximate gradients, which is not the case in the family of bundle methods. 
%------
These encouraging results motivate us to generalize the presented  algorithm to the class of nonconvex locally Lipschitz functions in the next work.

\bibliographystyle{spmpsci}
\bibliography{maleknia}

\end{document}